\documentclass[12pt,reqno]{amsart}

\usepackage{microtype}
\usepackage[margin=4cm]{geometry}

\usepackage[shortlabels]{enumitem}
\setlist[enumerate]{label={(\arabic*)}}

\usepackage{amssymb}
\usepackage[dvipsnames]{xcolor}
\usepackage
[colorlinks=true,linkcolor=Maroon,citecolor=OliveGreen,backref]
{hyperref}
\usepackage[abbrev, alphabetic]{amsrefs}  
\usepackage[capitalize]{cleveref} 
\crefname{equation}{}{}
\usepackage{tikz-cd}
\usepackage{tikz}

\hypersetup{
   breaklinks=true,   
   colorlinks=true,   
}

\usepackage[norefs, nocites]{refcheck}

\makeatletter
\newcommand{\refcheckize}[1]{%
  \expandafter\let\csname @@\string#1\endcsname#1%
  \expandafter\DeclareRobustCommand\csname relax\string#1\endcsname[1]{%
    \csname @@\string#1\endcsname{##1}\wrtusdrf{##1}}%
  \expandafter\let\expandafter#1\csname relax\string#1\endcsname
}
\makeatother

\refcheckize{\cref}
\refcheckize{\Cref}

\numberwithin{equation}{section}

\newtheorem{lemma}{Lemma}[section]
\newtheorem{theorem}[lemma]{Theorem}
\newtheorem{proposition}[lemma]{Proposition}
\newtheorem{corollary}[lemma]{Corollary}

\theoremstyle{definition}
\newtheorem{remark}[lemma]{Remark}

\newcommand\opr[1]{\operatorname{#1}}

\newcommand{\eps}{\epsilon}
\newcommand{\vareps}{\varepsilon}
\renewcommand{\subset}{\subseteq}
\renewcommand{\supset}{\supseteq}
\renewcommand{\bar}{\overline}

\def\sm{\smallsetminus}
\renewcommand{\setminus}{\sm}

\def\F{\mathbf{F}}

\def\Cl{\opr{Cl}}
\def\SL{\opr{SL}}
\def\GL{\opr{GL}}
\def\SU{\opr{SU}}
\def\Sp{\opr{Sp}}
\def\Or{\opr{O}}

\def\Aut{\opr{Aut}}
\def\Fix{\opr{Fix}}
\def\Sym{\opr{Sym}}

\newcommand\br[1]{{\left(#1\right)}}
\newcommand\floor[1]{\left\lfloor{#1}\right\rfloor}
\newcommand\ceil[1]{\left\lceil{#1}\right\rceil}
\newcommand\pfrac[2]{\br{\frac{#1}{#2}}}
\newcommand{\gen}[1]{\left\langle{#1}\right\rangle}

\def\nsgp{\trianglelefteq}

\def\tv{\mathfrak{T}}

\def\tr{\opr{tr}}
\def\im{\opr{im}}  
\def\ker{\opr{ker}}  

\def\Hom{\opr{Hom}}
\def\End{\opr{End}}
\def\Cay{\opr{Cay}}

\usepackage{todonotes}

\begin{document}

\title[Diameter with respect to transvections]{Diameter of classical groups generated by transvections}

\author{Sean Eberhard}
\address{Sean Eberhard, Mathematical Sciences Research Centre, Queen's University Belfast, Belfast BT7~1NN, UK}
\email{s.eberhard@qub.ac.uk}

\thanks{The author is supported by the Royal Society.}

\begin{abstract}
    Let $G$ be a finite classical group generated by transvections, i.e., one of $\SL_n(q)$, $\SU_n(q)$, $\Sp_{2n}(q)$, or $\Or^\pm_{2n}(q)~(q~\text{even})$,
    and let $X$ be a generating set for $G$ containing at least one transvection.
    Building on work of Garonzi, Halasi, and Somlai,
    we prove that the diameter of the Cayley graph $\Cay(G, X)$ is bounded by $(n \log q)^C$ for some constant $C$.
    This confirms Babai's conjecture on the diameter of finite simple groups in the case of generating sets containing a transvection.

    By combining this with a result of the author and Jezernik it follows that if $G$ is one of $\SL_n(q)$, $\SU_n(q)$, $\Sp_{2n}(q)$ and $X$ contains three random generators then with high probability the diameter $\Cay(G, X)$ is bounded by $n^{O(\log q)}$. This confirms Babai's conjecture for non-orthogonal classical simple groups over small fields and three random generators.
\end{abstract}
\maketitle


\section{Introduction}

This paper represents another step in the direction of the long-standing conjecture of Babai~\cite{BS92}*{Conjecture~1.7} on the diameter of finite simple groups with respect to arbitrary generating sets.
Given $g \in G = \gen X$, let $\ell_X(g)$ denote the minimal word length of $g$ with respect to $X \cup X^{-1}$, and for $S \subset G$ let $\ell_X(S) = \sup_{g \in S} \ell_X(g)$.
Babai's conjecture predicts that $\ell_X(G) \le (\log |G|)^{O(1)}$ for any (nonabelian) finite simple group $G$ and generating set $X$.
We call $\ell_X(G)$ the \emph{diameter} of $G$ with respect to $X$; it is equivalent to the diameter of the undirected Cayley graph $\Cay(G, X)$.

The past two decades have seen a slew of progress on Babai's conjecture.
In particular, following the seminal work of Helfgott~\cite{helfgott}, Breuillard--Green--Tao~\cite{BGT}, and Pyber--Szab\'o~\cite{pyber--szabo}, the conjecture is now completely resolved in bounded Lie rank.
In the case of alternating groups $A_n$, while Babai's conjecture predicts that the diameter should be bounded by a polynomial in $n$, a result of Helfgott and Seress~\cites{helfgott--seress,helfgott-unified} establishes a quasipolynomial bound.
Recently Bajpai, Dona, and Helfgott~\cite{BDH} proved a bound for classical Chevalley groups of large characteristic of the form $(\log |G|)^{O(r^5)}$, where $r$ is the rank of $G$.

Much stronger results are known about diameter with respect to typical generators (as opposed to worst-case generators, as in Babai's conjecture).
In this formulation we assume the generators are chosen uniformly at random and we are interested in results that fail with asymptotically negligible probability.
Polynomial (in $n$) bounds for the diameter of $A_n$ with respect to typical generators were established in \cites{babai--hayes, schlage-puchta, HSZ}.
More recently, it was proved in \cite{EJ} that if $G = \Cl_n(q)$ is a classical group with defining module $\F_q^n$ (so $\log |G| \approx c n^2 \log q$) then $\ell_X(G) \le q^2 n^{O(1)}$ with high probability provided that $X$ includes at least $q^C$ random generators.

Pyber has proposed%
\footnote{Personal communication. It was also previously emphasized in the introductions to \cite{H} and \cite{EJ}. The programme was inspired by \cite{BBS}.}
that Babai's conjecture should be broken down into three parts. Since the conjecture is resolved in bounded rank, we may assume that $G = A_n$ or $G = \Cl_n(q)$, $q = p^f$, that is, one of the groups $\SL_n(q)$, $\SU_n(\sqrt q)$, $\Sp_n(q)$, $\Omega_n^\vareps(q)$.
The \emph{degree} of an element $g \in G$ is the size of its support if $G = A_n$ or the rank of $g - 1$ if $G = \Cl_n(q)$.
Now let $X$ be an arbitrary generating set for $G$.
Pyber's programme consists of the following subproblems:
\begin{enumerate}
    \item Find a nontrivial element $g \in G$ of polylogarithmic length over $X$ whose degree is at most $(1-\eps)n$
    for some constant $\epsilon > 0$.
    \item Assuming $X$ includes an element of degree at most $(1-\epsilon)n$ as above, find a nontrivial element $g \in G$ of polylogarithmic length over $X$ whose degree is minimal in $G$.
    \item Assuming $X$ includes an element of minimal degree, show that
    every element of $G$ has polylogarithmic length over $X$.
\end{enumerate}
Part (1) is likely the most difficult.
Part (2) has been solved for $G = A_n$ for $\eps = 0.67$ in \cite{BBS} and more recently for $\eps = 0.37$ in~\cite{BGHHSS}
(doing the same for an arbitrary constant $\eps>0$ is a fascinating open problem).
Part (3) is actually trivial for $G = A_n$, since the number of $3$-cycles in $A_n$ is less than $n^3$, but for classical groups it is not at all trivial and in fact the focus of this paper.

Excluding the case of $\Omega_n^\eps(q)$ (which we leave to future work), the elements of minimal degree in a classical group $G = \Cl_n(q)$ are the transvections.
By recent work of Halasi and Garonzi--Halasi--Somlai~\cites{H, GHS} we have a suitable bound for $\ell_X(G)$ for $\SL_n(q)$, $\SU_n(\sqrt q)$, $\Sp_n(q)$ and any generating set $X$ containing a transvection, excluding all cases of characteristic $2$ and a few cases of characteristic $3$.

\begin{theorem}[Garonzi--Halasi--Somlai, \cite{GHS}*{Theorem~1.5}]
    Let $G$ be one of the non-orthogonal classical groups $\SL_n(q)$, $\SU_n(\sqrt q)$ ($q$ square), $\Sp_n(q)$ ($n$ even) with defining module $\F_q^n$. Let $X$ be a generating set for $G$ containing a transvection.
    Then
    \[
        \ell_X(G) \le (n \log q)^{O(1)}
    \]
    provided that
    \begin{itemize}
        \item $q$ is odd,
        \item $q \ne 9, 81$ if $G = \SL_n(q)$,
        \item $q \ne 81$ if $G = \SU_n(\sqrt q)$,
        \item $q \ne 9$ if $G = \Sp_n(q)$.
    \end{itemize}
\end{theorem}

In this paper we deal with the omissions above.
It is also natural to include the (imperfect) orthogonal groups in characteristic $2$.
Our main theorem is the following.

\begin{theorem}
    \label{thm:main}
    Let $G \le \SL_n(q)$ be an irreducible classical group generated by transvections, i.e., one of the following groups:
    \begin{align*}
        &\SL_n(q) && (n \ge 2),\\
        &\SU_n(\sqrt q) && (n \ge 3, q = q_0^2, (n,q_0) \ne (3,2)),\\
        &\Sp_n(q) && (n = 2m \ge 4),\\
        &\Or_n^\pm(q) && (n = 2m \ge 6, p = 2).
    \end{align*}
    Let $X$ be a generating set for $G$ containing a transvection.
    Then
    \[
        \ell_X(G) \le (n \log q)^{O(1)}.
    \]
\end{theorem}

It follows that Babai's conjecture \cite{BS92}*{Conjecture~1.7} holds for the simple groups $\opr{PSL}_n(q)$, $\opr{PSU}_n(\sqrt q)$, and $\opr{PSp}_n(q)$, as well as the index-2 overgroup $\opr{PO}_n^\pm(q)$ of the simple group $\opr{P\Omega}_n^\pm(q)$ for $q$ even, for generating sets containing the image of a transvection.
By combining with \cite{EJ}*{Theorem~1.1}, we also get the following corollary, which confirms Babai's conjecture for all non-orthogonal classical groups over small fields with respect to three random generators.
(As above, this is new only for $p=2$ and a few $p=3$ cases.)

\begin{corollary}
    Let $G$ be in one of the following intervals:
    \begin{align*}
        &\SL_n(q) \le G \le \GL_n(q),\\
        &\SU_n(\sqrt q) \le G \le \opr{GU}_n(\sqrt q),\\
        &G = \Sp_{2m}(q).
    \end{align*}
    Assume $\log q < c n / \log^2 n$ for a sufficiently small constant $c > 0$.
    Let $X = \{x,y,z\}$ where $x, y, z \in G$ are uniformly random.
    Then with probability $1 - e^{-cn}$ we have
    \[
        \ell_X(G) \le n^{O(\log q)}.
    \]
\end{corollary}

The proof of \Cref{thm:main} does not consist merely in reviewing \cite{GHS} and dealing with technicalities. There is a much greater variety of irreducible subgroups generated by transvections in characteristic $2$ than in odd characteristic, and these subgroups arise as obstructions in the proof.
While Garonzi, Halasi, and Somlai relied on more elementary results of Dickson and Humphries, our main tool is the following deep theorem of Kantor, building on previous work of Wagner, Piper, and Pollatsek.
The theorem classifies irreducible subgroups of $\SL_n(q)$ generated by transvections (and thus in particular gives additional context to \Cref{thm:main}).

\begin{theorem}[Kantor]
    \label{thm:classification}
    Let $n \ge 2$, $K = \F_q$, $q = p^f$, and let $G \le \SL_n(q)$ be an irreducible subgroup generated by transvections,
    not contained in a subfield subgroup.
    Then $G$ is one of the following up to conjugacy:
    \begin{enumerate}
        \item $\SL_n(q)$,
        \item $\SU_n(\sqrt q)$ where $q$ is a square,
        \item $\Sp_n(q)$ where $n$ is even,
        \item $\Or^\pm_n(q)$ where $n$ and $q$ are even,
        \item $C_a^{n-1} \rtimes S_n$ where $a > 1$ is odd and $q$ is the smallest power of $2$ such that $a \mid q-1$,
        \item $S_{n+1}$ or $S_{n+2}$ where $n$ is even and $q = 2$,
        \item $\SL_2(5) < \SL_2(9)$,
        \item $3 \cdot A_6 < \SL_3(4)$,
        \item $3 \cdot \opr{P\Omega}^{-,\pi}_6(3) < \SL_6(4)$.
    \end{enumerate}
\end{theorem}
\begin{remark}
    The proof of \Cref{thm:classification} is somewhat scattered in the literature.
    For $q$ odd see \cite{wagner} or \cite{ZS}.
    If $q$ is even then \cite{pollatsek}*{Theorem~1} (which depends on a result of Piper) asserts that $G \cong 3 \cdot A_6$ in $\SL_3(4)$ or the transvections in $G$ form a single conjugacy class of odd transpositions
    in the sense of Aschbacher.
    Now \cite{kantor}*{Theorem~II} (which depends on deep results of Fischer, Aschbacher, and Timmesfeld) delineates the remaining possibilities.
\end{remark}

In the rest of the paper, when $G$ is a irreducible linear group generated by transvections,
we say $G$ has \emph{linear}, \emph{unitary}, \emph{symplectic}, \emph{orthogonal}, \emph{monomial}, or \emph{symmetric} \emph{type} according to cases (1)--(6) respectively of \Cref{thm:classification}.
In the three remaining cases (which will not concern us) we say $G$ has \emph{exceptional type}.
We refer to the linear, unitary, symplectic, and orthogonal cases collectively as \emph{classical}.

In the course of the paper we build on many of the ideas and tools already developed in \cites{GHS,H}, particularly the transvection graph and its associated features such as symplectic and unitary cycles.
An effort has been made however to keep the present paper reasonably self-contained and it should not be necessary to refer to \cites{GHS,H} extensively.
Where a proposition in the present paper closely parallels a corresponding proposition in \cites{GHS,H}, a brief note is made to that effect.
Moreover, while we are free to assume $p \in \{2,3\}$ throughout,
doing so would not simplify the paper in any way, and simplifications are obtained for $p \ge 3$ (namely, we avoid many of the difficult calculations of \cite{GHS}*{Sections~4.2--4.3}).

\subsection*{Acknowledgments}

I am indebted to Laszlo Pyber, Bill Kantor, and Daniele Garzoni for comments and guidance.
I am also grateful to the anonymous referee for noticing an important oversight, as well as correcting a great many minor errors/typos.

\section{Bounded rank}

Babai's conjecture is known for simple groups of Lie type of bounded rank.
This is a corollary of the independent work of Breuillard--Green--Tao and Pyber--Szab\'o.

\begin{theorem}[Breuillard--Green--Tao~\cite{BGT}, Pyber--Szab\'o~\cite{pyber--szabo}]
    \label{thm:babai-bounded-rank}
    Let $G$ be a finite simple group of Lie type of rank $n$ and let $X$ be a generating set for $G$. Then
    \[
        \ell_X(G) \le O\pfrac{\log |G|}{\log |X|}^{O_n(1)}.
    \]
\end{theorem}

\Cref{thm:main} in bounded rank follows immediately from this and standard tools, as detailed below.

\begin{proposition}
    \label{main-thm-bounded-rank}
    In the situation of \Cref{thm:main} we have
    \[
        \ell_X(G) \le O\pfrac{\log |G|}{\log |X|}^{O_n(1)}.
    \]
\end{proposition}
\begin{proof}
    Assume $X = X^{-1}$ and $1 \in X$ without loss of generality.
    Note that $[G:G'] = 2$ in the orthogonal case and $G = G'$ in the other cases.
    Therefore by Schreier's lemma $Y = X^3 \cap G'$ is a generating set for $G'$.
    Moreover $|Y| \ge |X|/2$. Indeed, if $G' < G$, since $X$ generates $G$ there is some $x_0 \in X \sm G'$.
    Since $X \cap x_0 G' \subset x_0 (X^2 \cap G')$, we have
    \[
        |X| = |X \cap G'| + |X \cap x_0 G'| \le 2 |X^2 \cap G'| \le 2 |Y|.
    \]

    Applying \Cref{thm:babai-bounded-rank} to the simple quotient $S = G'/Z(G')$, it follows that $Y^m Z(G') = G'$, where
    \[
        m \le O(\log |S|  / \log |Y|)^{O_n(1)} \le O(\log |G| / \log |X|)^{O_n(1)}.
    \]
    In particular $|Y^m| \ge |G' / Z(G')|$.
    Applying Freiman's 3/2 theorem (see \cite{tao-book}*{Theorem~1.5.2 and Proposition~1.5.4})
    we have (for any $k \ge 1$) either $Y^{2^k m} = G'$ or
    \[
        |Y^{2^k m}| \ge (3/2)^k |Y^m| \ge (3/2)^k |G' / Z(G')|.
    \]
    If $G$ is symplectic or orthogonal we have $|Z(G')| \le 2$, so with $k = 2$ we must have $Y^{4m} = G'$ and thus $X^{12m+1} = G$.

    If $G$ is linear or unitary we have $G = G'$ and $|Z(G)| \le n$, so
    \[
        |Y^{2^k m}| \ge \min(|G|, (3/2)^k |G| / n).
    \]
    By Landazuri--Seitz~\cite{landazuri--seitz} and \cite{nikolov--pyber}*{Corollary~2.5}, the minimal degree of a nontrivial complex representation of $G$ is at least $|G|^{c/n} > q^{c' n}$ for some constants $c,c'>0$.
    Applying \cite{nikolov--pyber}*{Theorem~1}, we get $Y^{3 \cdot 2^k m} = G$ provided that $(3/2)^k / n > q^{-c' n / 3}$.
    It suffices to take $k$ to be a suitable constant depending on $c'$.
    Thus $Y^{m'} = X^{3m'} = G$ for some $m' \le O(\log |G| / \log |X|)^{O_n(1)}$.
\end{proof}

This special case is used in the proof of \Cref{thm:main} in two essentially different ways.
First, it allows us to assume that $n$ is larger than any fixed constant, which enables to avoid low-rank pathologies.
Second, in the course of the proof of \Cref{thm:main} we will apply \Cref{main-thm-bounded-rank} to a wide variety of bounded-rank subgroups.

\section{The transvection graph}

Following \cites{H,GHS} we define a graph based on a collection of transvections. To motivate the definition we first recall the concept of the transposition graph (which we will also use).

If $T \subset \Sym(\Omega)$ is a set of transpositions, the \emph{transposition graph} $\Gamma(T)$ is the graph with vertex set $\Omega$ and edge set $\{ij : (ij) \in T\}$.

\begin{proposition}
    \label{prop:transposition-graph}
    Let $T \subset \Sym(\Omega)$ be a set of transpositions and let $G = \gen T$.
    Then $G$ is the direct product of the full symmetric groups on the components of $\Gamma(T)$.
    In particular the following are equivalent:
    \begin{enumerate}
        \item $\Gamma(T)$ is connected;
        \item $G$ is transitive;
        \item $G = \Sym(\Omega)$.
    \end{enumerate}
\end{proposition}
\begin{proof}
    First assume $\Gamma(T)$ is connected. We will prove $G = \Sym(\Omega)$ by induction on $n$, the result being obvious if $n \le 1$.
    We may assume $\Gamma(T)$ is a tree by replacing $T$ with a subset.
    Let $i$ be a leaf of $\Gamma(T)$, let $t_0 = (ij) \in T$, and let $T_0 = T \sm \{t_0\}$.
    Then $\Gamma(T_0)$ can be identified with a tree with vertex set $\Omega_0 = \Omega \sm \{i\}$,
    so by induction $G_0 = \gen{T_0} = \Sym(\Omega_0)$, and it follows easily that
    $G = \gen {G_0, (ij)} = \Sym(\Omega)$.

    In general, let $\Omega_1, \dots, \Omega_k$ be the components of $\Gamma(T)$ and let $T_i \subset T$ be the set of transpositions in $T$ supported on $\Omega_i$. Then $T = T_1 \cup \cdots \cup T_k$ and the elements of $T_i$ commute with the elements of $T_j$ for $i \ne j$, so $G = G_1 \times \cdots \times G_k$ where $G_i = \gen{T_i}$.
    By the connected case, $G_i = \Sym(\Omega_i)$ for each $i$.
\end{proof}

Now let $K$ be a field of characteristic $p>0$ and let $V \cong K^n$ be a vector space over $K$ of finite dimension $n \ge 2$.
A \emph{transvection} is an element $t \in \SL(V)$ such that $t-1$ is a rank-one map.
Let $\tv \subset \SL(V)$ be the set of transvections.
Each transvection $t \in \tv$ has the form
\[
    t = 1 + v \otimes \phi \qquad (v \in V, \phi \in V^*, \phi(v) = 0).
\]
We make $\tv$ into a directed graph $\Gamma(\tv)$ by declaring there to be a directed edge $t \to s$ if and only if ${(t-1)(s-1) \ne 0}$.
Put another way, there is an edge $1 + v \otimes \phi \to 1 + u \otimes \psi$ if and only if $\phi(u) \ne 0$.
For any subset $T \subset \tv$ we denote by $\Gamma(T)$ the corresponding induced subgraph.
We also define linear subspaces $V(T) \le V$ and $V^*(T) \le V^*$ by
\begin{align*}
    V(T) &= \gen{v : 1 + v \otimes \phi \in T} = \sum_{t \in T} \im(t-1) = [V, G],\\
    V^*(T) &= \gen{\phi : 1 + v \otimes \phi \in T} = \sum_{t \in T} \ker(t-1)^\perp = (V^G)^\perp,
\end{align*}
where $G = \gen T$.
Here $[V,G]$ denotes the commutator of $V$ and $G$ as subgroups of $V \rtimes G$, and $V^G$ denotes the space of $G$-invariants in $V$.

Let $T \subset \tv$ be an arbitrary set of transvections and let $G = \gen T$.
In the rest of this section we describe how important properties of $G$ can be detected through careful examination of the transvection graph $\Gamma(T)$ as well as the subspaces $V(T)$ and $V^*(T)$.
We start with irreducibility of $G$.

\begin{proposition}[\cite{GHS}*{Theorem~3.1}]
    \label{prop:irreducible}
    The group $G \le \SL(V)$ is irreducible if and only if
    \begin{enumerate}
        \item $V(T) = V$,
        \item $V^*(T) = V^*$,
        \item $\Gamma(T)$ is strongly connected.
    \end{enumerate}
\end{proposition}
\begin{proof}
    \emph{Necessity:}
    Suppose $G$ is irreducible. Then $[V,G] = V$ and $V^G = 0$, so (1) and (2) hold.
    Now suppose $S$ is a nonempty subset of $T$ with no incoming edges in $\Gamma(T)$.
    Then, for $t \in T$ and $s \in S$,
    \begin{equation}
        \label{eq:no-incoming-edges}
        (t-1)(s-1) \ne 0 \implies t \in S.
    \end{equation}
    Let $U = V(S)$. Then $U$ is nonzero and $\gen S$-invariant, and the above equation shows that $T \sm S$ acts trivially on $U$, so $U$ is $G$-invariant, so $U = V$.
    Thus for every $t \in T$ there is some $s \in S$ such that $(t-1)(s-1) \ne 0$, whence $t \in S$ by \eqref{eq:no-incoming-edges}.
    Hence (3) holds.

    \emph{Sufficiency:} Let $U \le V$ be a nonzero $G$-invariant subspace. Let $S \subset T$ be the set of $s$ such that $\im(s-1) \le U$.
    Let $u \in U \sm \{0\}$.
    By (2) there is at least one $t \in T$ such that $(t-1) u \ne 0$.
    By $G$-invariance of $U$ it follows that $\im(t - 1) \le U$, so $t \in S$.
    Hence $S$ is nonempty, and the same calculation shows that $S$ has no incoming edges, so by (3) we have $S = T$. Thus by (1) we have $U = V$. Hence $G$ is irreducible.
\end{proof}

We say $G$ \emph{defined over} a subfield $L \le K$ if for some basis of $V$ the elements of $G$ are represented by matrices with entries in $L$.
The \emph{defining field} of $G$ is the smallest field such that $G$ is defined over $L$.

Given transvections $t_1, \dots, t_k \in \tv$, the \emph{weight} of the tuple $(t_1, \dots, t_k)$ is defined as
\begin{equation}
    \label{eq:weight-trace}
    w(t_1, \dots, t_k) = \tr((t_1 - 1) \cdots (t_k - 1)).
\end{equation}
If $t_i = 1 + v_i \otimes \phi_i$ for each $i$ then
\[
    w(t_1, \dots, t_k) = \phi_1(v_2) \phi_2(v_3) \cdots \phi_k(v_1).
\]
Note that $w(t_1, \dots, t_k) \ne 0$ if and only if $t_1, \dots, t_k$ is a directed cycle in $\Gamma(\tv)$.

If $T \subset \tv$ we denote by $L_k(T)$ the subfield generated by the weights of the cycles of $\Gamma(T)$ of length at most $k$.
Let $L(T) = \bigcup_{k=1}^\infty L_k(T)$.

\begin{proposition}
    [cf.~\cite{GHS}*{Section~3.2}]
    Assume $G$ is irreducible.
    Then $L(T)$ is the defining field of $G$.
\end{proposition}
\begin{proof}
    It follows from \eqref{eq:weight-trace} (and induction on $k$) that $L_k(T)$ is equal to the field generated by $\{\tr(g) : \ell_T(g) \le k\}$,
    so $L(T) = \bigcup_{k=1}^\infty L_k(T)$ is the field generated by $\{\tr(g) : g \in G\}$ (the \emph{trace field} of $G$).
    It is well-known that the latter is precisely the defining field of $G$ in positive characteristic
    (see \cite{isaacs-book}*{Corollaries~9.5(c) and 9.23}).
\end{proof}

Again suppose $t_1,\dots,t_k \in \tv$. We define
\[
    d_s(t_1, \dots, t_k) = w(t_1, \dots, t_k) - (-1)^k w(t_k, \dots, t_1)
\]
and we call the tuple $(t_1, \dots, t_k)$ \emph{symplectic} if $d_s = 0$.
Similarly, if $\theta$ is an involution of $K$, we define
\[
    d_\theta(t_1, \dots, t_k) = w(t_1, \dots, t_k) - (-1)^k w(t_k, \dots, t_1)^\theta,
\]
and we call the tuple $(t_1, \dots, t_k)$ $\theta$-\emph{unitary} if $d_\theta = 0$.
(If $K$ is the field $\F_q$ where $q$ is a square then there is a unique involution of $K$ and we may drop the ``$\theta$-'' prefix.)
Note that if $(t_1, \dots, t_k)$ is symplectic or $\theta$-unitary then it is a cycle in $\Gamma(\tv)$ if and only if it is a two-way cycle.

\begin{proposition}
    [cf.~\cite{GHS}*{Theorem~3.12}]
    \label{prop:invariant-form}
    Assume $G$ is irreducible. Let $D$ be the directed diameter of $\Gamma(T)$.
    \begin{enumerate}
        \item There is a $G$-invariant symplectic form if and only if every cycle of $\Gamma(T)$ of length at most $2D+1$ is symplectic.
        \item There is a $G$-invariant $\theta$-unitary form if and only if every cycle of $\Gamma(T)$ of length at most $2D+1$ is $\theta$-unitary.
    \end{enumerate}
\end{proposition}
\begin{proof}
    We can prove both statements at once by seeking a $G$-invariant nondegenerate sesquilinear form $f$ such that $f(x, y) + f(y, x)^\theta = 0$ identically, where $\theta$ is an automorphism of $K$ of order at most $2$.
    If $\theta = 1$ and $p = 2$ then we additionally require $f(x, x) = 0$ identically.
    If $\theta$ is nontrivial we obtain a unitary form by multiplying $f$ by a nonzero scalar $\lambda$ such that $\lambda^\theta = - \lambda$.
    We use the unitary language for both cases.

    If $G$ preserves such a form $f$ then $V$ and $V^*$ can be identified by defining $v^*(u) = f(u, v)$ for all $u, v \in V$, and the resulting map $v \mapsto v^*$ is a $\theta$-semilinear isomorphism of $V$ with $V^*$.
    The transvections in $\Aut(f)$ are precisely the transvections of the form $1 + \lambda v \otimes v^*$ with $\lambda \in \Fix(\theta)$ and $f(v, v) = 0$. Thus the weight of a cycle $t_1, \dots, t_k$, where $t_i = 1 + \lambda_i v_i \otimes v_i^*$, is
    \[
        w(t_1, \dots, t_k)
        = \lambda_1 \cdots \lambda_k f(v_2, v_1) \cdots f(v_1, v_k)\\
        = (-1)^k w(t_k, \dots, t_1)^\theta,
    \]
    so every cycle is unitary.

    Conversely suppose every cycle in $\Gamma(T)$ of length at most $2D+1$ is unitary.
    For each $t \in T$ fix a representation $t = 1 + v_t \otimes \phi_t$.
    We claim we can choose nonzero scalars $\lambda_t \in \Fix(\theta)$ for each $t \in T$ such that
    \begin{equation}
        \label{eq:lambda-coherence}
        \lambda_t \phi_t(v_s) + \lambda_s \phi_s(v_t)^\theta = 0 \qquad (t,s \in T).
    \end{equation}
    Indeed, first note that every edge $(t, s)$ in $\Gamma(T)$ is two-way since every edge is contained in a cycle of length at most $D+1$, and every such cycle is unitary.
    Now for any path $\gamma = (t_0, t_1, \dots, t_k)$ in $\Gamma(T)$ of length at most $D+1$ define
    \[
        \lambda_\gamma = (-1)^k \frac{\br{\phi_0(v_1) \cdots \phi_{k-1}(v_k)}^\theta}{\phi_k(v_{k-1}) \cdots \phi_1(v_0)}.
    \]
    The cycle condition ensures that $\lambda_\gamma = \lambda_{\gamma'}^\theta$ whenever $\gamma$ and $\gamma'$ have the same endpoints and combined length at most $2D+1$.
    In particular $\lambda_\gamma \in \Fix(\theta)$ if $\gamma$ has length at most $D$.
    Choose $t_0 \in T$ arbitrarily and define $\lambda_t$ to be the common value of $\lambda_\gamma$ for paths $\gamma$ from $t_0$ to $t$ of length at most $D$. Then $(\lambda_t)$ is the required global solution to \eqref{eq:lambda-coherence}.

    Now define $L : K^T \to V \times V^*$ by
    \[
        L(x) = (L_1(x), L_2(x)) = \br{\sum_{t \in T} x_t v_t, \sum_{t \in T} x_t^\theta \lambda_t \phi_t}.
    \]
    By \eqref{eq:lambda-coherence} we have
    \[
        \lambda_t \phi_t(L_1(x)) + L_2(x)(v_t)^\theta = 0 \qquad (t \in T, x \in K^T).
    \]
    In particular, since $V(T) = V$ and $V^*(T) = V^*$, we have $L_1(x) = 0$ if and only if $L_2(x)
     = 0$.
    Thus $\im(L) \subset V \times V^*$ is the graph of a $\theta$-semilinear isomorphism $*:V \to V^*$ such that $v_t^* = \lambda_t \phi_t$ for all $t \in T$.
    Define $f(u, v) = v^*(u)$ for $u, v \in V$.
    Then $f$ is sesquilinear and nondegenerate, and by \eqref{eq:lambda-coherence} we have
    \[
        f(v_s, v_t) + f(v_t, v_s)^\theta = 0 \qquad (t, s \in T),
    \]
    which implies that $f$ is hermitian antisymmetric since $\im(T) = V$.
    Moreover for every $t \in T$ we have
    \[
        f(v_t, v_t) = v_t^*(v_t) = \lambda_t \phi_t(v_t) = 0.
    \]
    In particular if $p = 2$ and $\theta = 1$ then $f$ is alternating.
    Finally, since each $t \in T$ has the form $t = 1 + v_t \otimes \phi_t = 1 + \lambda_t^{-1} v_t \otimes v_t^*$ with $\lambda_t \in \Fix(\theta)$, $f$ is $G$-invariant, as required.
\end{proof}

Finally, we explain how to determine from the transvection graph whether $G$ is contained in an orthogonal group. Note that this is only relevant in characteristic $2$.

\begin{proposition}
    \label{prop:orthogonal}
    Assume $K$ is a perfect field of characteristic $2$.
    Assume $G$ is irreducible and preserves a symplectic form $f$.
    Let $v \mapsto v^*$ be the isomorphism $V \to V^*$ defined by $v^*(u) = f(u,v)$.
    Then for each $t \in T$ there is a unique $v \in V$ such that $t = 1 + v \otimes v^*$,
    and $G$ preserves a quadratic from associated to $f$ if and only if for every linear relation
    \begin{equation}
        \label{eq:linear-relation}
        \sum_{i=1}^m \lambda_i v_i = 0 \qquad (t_i = 1 + v_i \otimes v_i^* \in T~\text{for}~i = 1, \dots, m)
    \end{equation}
    with $m \le n+1$
    we have the corresponding quadratic relation
    \begin{equation}
        \label{eq:quadratic-relation}
        \sum_{i=1}^m \lambda_i^2 + \sum_{1 \le i < j \le m} \lambda_i \lambda_j f(v_i, v_j) = 0.
    \end{equation}
\end{proposition}
\begin{proof}
    We noted in the proof of the previous proposition that each transvection $t \in \tv \cap \Aut(f)$ can be written $t = 1 + \lambda v \otimes v^*$ with $\lambda \in K$.
    Since $K$ is perfect, $\lambda = \mu^2$ for some $\mu \in K$, so $t = 1 + u \otimes u^*$ where $u = \mu v$. Moreover if $u \otimes u^* = w \otimes w^*$ ($w \in V$) then $u = \lambda w$ for some $\lambda$ such that $\lambda^2 = 1$, whence $\lambda = 1$. Therefore each $t \in T$ has the form $1 + v_t \otimes v_t^*$ and $v_t$ is determined uniquely by $t$.

    Suppose $Q$ is a quadratic form associated to $f$. By a well-known direct calculation, a transvection $t = 1 + v \otimes v^*$ preserves $Q$ if and only if $Q(v) = 1$.
    Suppose $G$ preserves a quadratic form $Q$. Then $Q(v) = 1$ for all $t = 1 + v \otimes v^* \in T$.
    By applying $Q$ to \eqref{eq:linear-relation} we get \eqref{eq:quadratic-relation}.

    Conversely, suppose \eqref{eq:linear-relation} implies \eqref{eq:quadratic-relation} identically.
    Since $G$ is irreducible, $V(T) = V$ by \Cref{prop:irreducible}, so there is a basis $v_1, \dots, v_n$ among the vectors $v$ such that $1 + v \otimes v^* \in T$. There is a unique quadratic form $Q$ associated to $f$ such that $Q(v_i) = 1$ for $i = 1, \dots, n$. Now, if $1 + v \otimes v^* \in T$, because $v_1, \dots, v_n$ is a basis we have some linear relation
    \[
        v + \sum_{i=1}^n \lambda_i v_i = 0.
    \]
    To this linear relation we can apply either \eqref{eq:linear-relation} $\implies$ \eqref{eq:quadratic-relation} or we can apply the quadratic from $Q$.
    By comparing the results we find that $Q(v) = 1$.
    Thus $Q(v) = 1$ for all $1 + v \otimes v^* \in T$, so $Q$ is $G$-invariant.
\end{proof}

\section{Some splitting results}

Continue to assume $K$ is a field of characteristic $p > 0$ and $V \cong K^n$ is a vector space of dimension $n \ge 2$.
As before $\tv \subset \SL(V)$ denotes the set of transvections, $T \subset \tv$, and $G = \gen T$.

By \Cref{prop:irreducible}, $G$ is irreducible if and only if $\Gamma(T)$ is strongly connected and $V(T) = V$ and $V^*(T) = V^*$.
In this section we consider the situation in which $\Gamma(T)$ is strongly connected but one or other of the latter conditions may fail.

Throughout this section,
\[
    U = V(T) = [V,G], \qquad W = V(T) \cap V^*(T)^\perp = [V,G]^G.
\]
Observe that $U$ and $W$ are $G$-invariant subspaces, so $G$ acts naturally on the section $U/W$.
Let $g \mapsto \bar g$ denote the homomorphism $G \to \SL(U/W)$.

\begin{proposition}
    \label{prop:irreducible-section}
    Assume $\Gamma(T)$ is strongly connected.
    Then $\bar T$ is a set of transvections of $U/W$ and the map $t \mapsto \bar t$ defines a graph homomorphism $\Gamma(T) \to \Gamma(\bar T)$ that is surjective on both vertices and edges.
    Moreover, the weights of corresponding cycles are equal.
    In particular, $\bar G$ is irreducible.
\end{proposition}

\begin{proof}
    If $t \in T$, the strong connectedness of $\Gamma(T)$ ensures the existence of $s, s' \in T$ such that $(s-1)(t-1) \ne 0$ and $(t-1)(s'-1) \ne 0$.
    Thus $\im(t-1)$ is not contained in $W$ and $\ker(t-1)$ does not contain $U$,
    so $\bar t - 1$ is a nontrivial rank-one map of $U / W$.
    Thus $\bar T$ consists of transvections of $U / W$.
    Moreover if $(t, s)$ is an edge of $\Gamma(T)$ then $(t-1)(s-1)$ also descends to a nontrivial rank-one map of $U/W$, so $(\bar t, \bar s)$ is an edge of $\Gamma(T)$,
    and conversely it is clear that if $(\bar t -1)(\bar s - 1)$ is nonzero so must be $(t-1)(s-1)$.
    Thus $t \mapsto \bar t$ induces a graph homomorphism $\Gamma(T) \to \Gamma(\bar T)$ that is surjective on both vertices and edges.

    Suppose $(t_1, \dots, t_n)$ is a cycle of $\Gamma(T)$. Then $w(t_1, \dots, t_n) = \tr(\alpha)$ where $\alpha = (t_1-1) \cdots (t_n-1)$.
    Observe that $\alpha$ is a rank-one map such that $\im(\alpha) \le U$ and $\alpha(W) = 0$. Therefore $\tr(\alpha) = \tr(\bar \alpha)$, where $\bar \alpha$ is the image of $\alpha$ in $\End(U/W)$, so $w(t_1, \dots, t_n) = w(\bar t_1, \dots, \bar t_n)$.

    Finally, $V(\bar T) = U / W$ and $V^*(\bar T)^\perp = 0$.
    Since $\Gamma(T)$ is strongly connected, so its homomorphic image $\Gamma(\bar T)$.
    Thus $\bar G = \gen{\bar T}$ is irreducible by \Cref{prop:irreducible}.
\end{proof}

We say $T \subset \tv$ is \emph{nondegenerate} if the natural pairing
\[
    V(T) \times V^*(T) \to K
\]
is nondegenerate.
This is equivalent to the presence of a direct sum decomposition
\[
    V = V(T) \oplus V^*(T)^\perp = [V,G] \oplus V^G.
\]
We say $T$ is \emph{weakly nondegenerate} if the natural pairing has either zero left kernel or zero right kernel.
Observe that $T$ is nondegenerate if and only if it is weakly nondegenerate and $\dim V(T) = \dim V^*(T)$.

\begin{proposition}
    Suppose $G$ preserves a nondegenerate sesquilinear form $f$.
    Then $\dim V(T) = \dim V^*(T)$. In particular if $T$ is weakly nondegenerate then it is nondegenerate.
\end{proposition}
\begin{proof}
    For $v \in V$ let $v^* \in V^*$ be defined by $v^*(u) = f(u, v)$.
    The transvections in $\Aut(f)$ all have the form $1 + \lambda v \otimes v^*$ for some $\lambda \in K$.
    Since $v \mapsto v^*$ is semilinear it follows that $V^*(T) = \{v^* : v \in V(T)\}$ and $\dim V(T) = \dim V^*(T)$.
\end{proof}

\begin{proposition}
    Assume
    \begin{enumerate}
        \item $\Gamma(T)$ is strongly connected,
        \item $\bar G = \SL(U)$,
        \item $\dim U \ge 4$.
    \end{enumerate}
    Then
    $G \cong \SL(U) \ltimes U^d$ for some $d \ge 0$ and there is a basis for $V$ such that
    \[
        G = \left\{
            \begin{pmatrix}
                \bar g & u & 0 \\
                0 & I & 0 \\
                0 & 0 & I
            \end{pmatrix}
            : \bar g \in \SL(U), u \in U^d
        \right\}.
    \]
\end{proposition}

\begin{proof}
    Choose a basis for $V$ extending a basis for $U$.
    Since $G$ preserves $U = [V,G]$ and acts trivially on $V/U$, the elements of $G$ take the form
    \[
    \begin{pmatrix}
        * & * \\
        0 & I
    \end{pmatrix}.
    \]
    The elements of $M = C_G(U)$ have the form
    \[
    \begin{pmatrix}
        I & * \\
        0 & I
    \end{pmatrix},
    \]
    and $M$ can be identified with a subspace of $\Hom(V/U, U) \cong U^m$, where $m = \dim (V/U)$.
    Moreover the conjugation action of $G$ on $M$ corresponds to the natural action of $G$ on $U^m$.

    Since $\dim U \ge 4 > 1$ and $\bar G = \SL(U)$, $G$ acts transitively on $U \sm \{0\}$.
    In particular, $U$ is an irreducible $kG$-module where $k = \F_p$ is the prime subfield of $K$.

    Suppose $U_1$ is an irreducible $kG$-submodule of $U^m$.
    Then $U_1 \cong U$. Let $f : U \to U_1$ be an isomorphism.
    By composing with any of the $m$ coordinate projections $\pi_i:U^m \to U$ we get a map $\pi_i \circ f \in \End_{kG}(U)$.
    It is easy to check that $\End_{kG}(U) = K$ using the fact that each $\alpha \in \End_{kG}(U)$ must commute with every transvection of $U$.\footnote{Alternatively, by Schur's lemma, $\End_{kG}(U)$ is a division algebra containing $K$. If $K$ is finite, Wedderburn's little theorem implies that $\End_{kG}(U)$ is a field and therefore $\End_{kG}(U) = \End_{KG}(U) = K$ since $U$ is an absolutely irreducible $KG$-module.}
    Therefore there is some $a_i \in K$ such that $\pi_i \circ f(u) = a_i u$ for all $u \in U$.
    It follows that $f(u) = (a_1 u, \dots, a_m u)$ for all $u \in U$.
    This shows that all irreducible $kG$-submodules of $U^m$ are equivalent up to the natural action of $\GL_m(K)$.
    By induction on dimension, it follows that every $kG$-submodule of $U^m$ is equivalent to $\{0\}^{m-d} \oplus U^d$ for some $d$.

    Since $M$ can be identified with a $kG$-submodule of $U^m$, it follows from the previous paragraph that we can choose the basis so that
    \[
    M = \left\{\begin{pmatrix}
        I & 0 & u \\
        0 & I & 0 \\
        0 & 0 & I
    \end{pmatrix}
    : u \in U^d
    \right\}
    \]
    for some $d \le m$.
    Therefore
    \[
    G = \left\{
    \begin{pmatrix}
        \bar g & f(\bar g) & u \\
        0 & I & 0 \\
        0 & 0 & I
    \end{pmatrix}
    : \bar g \in \SL(U), u \in U^d
    \right\},
    \]
    for some well-defined map $f : \SL(U) \to U^{m - d}$. By direct calculation, $f$ must be a crossed homomorphism.

    By a result of Higman~\cite{higman}*{Lemma~4}, $H^1(\SL(U), U) = 0$ for $\dim U \ge 4$ (see also \cites{pollatsek-H^1,cline-parshall-scott}).
    Therefore $f(\bar g) = \bar g u - u$ for some $u \in U^{m - d}$, and by conjugating appropriately we can arrange that $f=0$.
\end{proof}

\begin{lemma}
    Let $H = \SL(V) \ltimes V^d$ and let $h \mapsto \bar h$ be the natural map from $H$ to $\bar H = \SL(V)$.
    Let $X$ be a generating set for $H$. Then
    \[
        \ell_X(H) \le 6^d \ell_{\bar X}(\bar H).
    \]
\end{lemma}

\begin{proof}
    Assume without loss of generality $X = X^{-1}$ and $1 \in X$.
    We may also assume that $\ell_{\bar X}(\bar H) < \infty$.
    Let $W = V^d$ and let $U \cong V$ be an irreducible submodule of $W$.
    Let $H^* = H / U \cong \SL(V) \rtimes V^{d-1}$.
    By induction on $d$, $L = \ell_{X^*}(H^*) < \infty$.
    Therefore $X^L$ covers $H^*$ and $\bar H$ a fortiori.
    If $X^{L+1} = X^L$ then $X^L = H$.
    Otherwise there is some $x \in X^{L+1} \sm X^L$.
    Since $X^L$ covers $H^*$, there is some $y \in X^L$ such that $u = xy^{-1} \in U$, and $u$ is nontrivial since $x \ne y$.
    Now since $\bar H$ is transitive on $U \sm \{0\}$ and $X^L$ covers $\bar H$ we have
    \[
        U \sm \{0\} \subset u^{X^L} \subset X^{4L+1}.
    \]
    Since $X^L$ covers $H^* \cong H / U$ it follows that $H \subset X^{5L+1} \subset X^{6L}$. Thus $\ell_X(H) \le 6 \ell_{X^*}(H^*)$,
    and the claimed bound holds by induction.
\end{proof}

\begin{proposition}
    \label{prop:tv-in-SL-sections}
    Assume
    \begin{enumerate}
        \item $\Gamma(T)$ is strongly connected,
        \item $T$ is weakly nondegenerate,
        \item $\bar G = \SL(U/W)$,
        \item $\min(\dim V(T), \dim V^*(T)) \ge 4$.
    \end{enumerate}
    Then
    \[
        G \cap \tv = \{1 + v \otimes \phi : v \in V(T), \phi \in V^*(T)\}.
    \]
    Moreover $\ell_T(G) \le 6^d \ell_{\bar T}(\bar G)$, where
    \[
        d = | \dim V(T) - \dim V^*(T) | \le |T|.
    \]
\end{proposition}
\begin{proof}
    For $v \in V$ and $\phi \in V^*$, let $\langle v, \phi \rangle = \phi(v)$.
    Then $\langle,\rangle:V\times V^* \to K$ is a perfect pairing and for each $g \in \SL(V)$ there is a corresponding element $g^* \in \SL(V^*)$ defined by
    \[
        \langle gv, \phi\rangle = \langle v, g^*\phi\rangle.
    \]
    Moreover $g \mapsto (g^*)^{-1}$ is isomorphism $\SL(V) \to \SL(V^*)$ which exchanges $V(T)$ and $V^*(T)$.
    To be precise, if $V_1 = V^*$ then $V_1^* = V$ naturally, and $T^* = \{(t^*)^{-1} : t \in T\}$ is a set of transvections in $\SL(V_1)$ such that $V_1(T^*) = V^*(T)$ and $V_1^*(T^*) = V(T)$.
    Therefore, up to replacing $G$ by $G^* = \{g^* : g \in G\}$, we can assume that $V(T) \cap V^*(T)^\perp = 0$, i.e., $W = 0$.
    Now the previous proposition applies, so $V = X \oplus V^G$ for some $G$-invariant complement $X \le V$ to $V^G$ (containing $U$) and
    \[
        G = \{g \in \SL(V) : (g-1)(X) \le U, (g-1)(V^G) = 0\}.
    \]
    In particular, if $t = 1 + v \otimes \phi \in \tv$, then $t \in G$ if and only if $v \in U = V(G)$ and $\phi(V^G) = 0$, i.e., $\phi \in (V^G)^\perp = V^*(T)$.

    The last part follows from $G \cong \bar G \ltimes U^d$ and the previous lemma.
\end{proof}

\section{Path shortening and other tricks}

The method of this section is inspired by \cite{GHS}*{Lemma~4.1} and \cite{H}*{Lemma~2.4}.

Given a vector $u \in V$ and a transvection $t = 1 + v \otimes \phi \in \tv$,
let us write $t \to u$ if $\phi(u) \ne 0$.
Similarly given $\psi \in V^*$ let us write $\psi \to t$ if $\psi(v) \ne 0$.
This notation is compatible with the corresponding notation for the transvection graph: if $t = 1 + v \otimes \phi$ and $s = 1 + u \otimes \psi$ then
\[
    t \to s \iff t \to u \iff \phi \to s \iff \phi(u) \ne 0.
\]

Call $T \subset \tv$ \emph{dense} if for every nonzero $v \in V$ and nonzero $\phi \in V^*$ there is some $t \in T$ such that
\[
    \phi \to t \to v.
\]
Clearly if $T$ is dense then $V(T) = V$, $V^*(T) = V^*$, and $\Gamma(T)$ is strongly connected (in fact $\Gamma(T)$ has directed diameter at most $2$), so $G = \gen T$ is irreducible by \Cref{prop:irreducible}.
The following useful proposition is a partial converse.

\begin{proposition}
    \label{prop:density}
    Let $T \subset \tv$ be a symmetric set of transvections such that $\gen T$ is irreducible.
    Then $T^{2n-1} \cap \tv$ is dense.
\end{proposition}
\begin{proof}
    Let $v \in V$ and $\phi \in V^*$ be both nonzero.
    By \Cref{prop:irreducible}, $V(T) = V$, $V^*(T) = V^*$, and $\Gamma(T)$ is strongly connected.
    It follows that there are transvections $t_1, \dots, t_m \in T$ for some $m \ge 1$ such that
    \[
        \phi \to t_1 \to \cdots \to t_m \to v.
    \]
    Write $t_i = 1 + v_i \otimes \phi_i$ for $i = 1, \dots, m$.
    Assume $m$ is minimal for the existence of such $t_1, \dots, t_m$.
    If $m = 1$ we are done, so assume $m > 1$.
    By minimality of $m$,
    \begin{enumerate}
        \item $\phi \not\to t_2, \dots, t_m$,
        \item $t_1, \dots, t_{m-1} \not\to v$,
        \item $t_i \not\to t_j$ for $j - i > 1$.
    \end{enumerate}
    In particular $v_1, \dots, v_m$ must be linearly independent, so $m \le n$.
    Now consider
    \[
        t' = t_1 \cdots t_{m-1} t_m t_{m-1}^{-1} \cdots t_1^{-1} = 1 + v' \otimes \phi',
    \]
    where
    \begin{align*}
        v' &= t_1 \circ \cdots \circ t_{m-1} (v_m), \\
        \phi' &= \phi_m \circ t_{m-1}^{-1} \circ \cdots \circ t_1^{-1}.
    \end{align*}
    Using properties (1) and (3) above, we have
    \begin{align*}
        \phi(v')
        &= \phi (1 + v_1 \otimes \phi_1) \cdots (1 + v_{m-1} \otimes \phi_{m-1}) (v_m)\\
        &= \phi(v_1) \phi_1(v_2) \cdots \phi_{m-1}(v_m) \ne 0.
    \end{align*}
    Similarly, by property (2),
    \begin{align*}
        \phi'(v)
        &= \phi_m (1 - v_{m-1} \otimes \phi_{m-1}) \cdots (1 - v_1 \otimes \phi_1) (v) \\
        &= \phi_m (v) \ne 0.
    \end{align*}
    Thus $\phi \to t' \to v$. Since $t' \in T^{2m-1} \cap \tv$, we are done.
\end{proof}

\begin{proposition}
    \label{prop:connect-up}
    Let $T \subset \tv$ be a dense set of transvections.
    Let $T_0 \subset T$ be a nonempty subset and assume $\Gamma(T_0)$ has $k$ strongly connected components.
    Then there is a set of transvections $T_1 \subset T$ such that
    \begin{enumerate}
        \item $T_0 \subset T_1$,
        \item $|T_1| \le |T_0| + k$,
        \item $\Gamma(T_1)$ is strongly connected.
    \end{enumerate}
    If $G = \gen T$ preserves a symplectic or unitary form then we can assume (2') $|T_1| \le |T_0| + k - 1$.
\end{proposition}
\begin{proof}
    Let $t_1, \dots, t_k \in T_0$ be representatives for the $k$ strongly connected components of $\Gamma(T_0)$. By density of $T$ there are transvections $t'_1, \dots, t'_k \in T$ such that
    \[
        t_1 \to t'_1 \to t_2 \to t'_2 \to \cdots \to t'_{k-1} \to t_k \to t'_k \to t_1.
    \]
    Let $T_1 = T_0 \cup \{t'_1, \dots, t'_k\}$.
    Then $|T_1| \le |T_0| + k$ and $\Gamma(T_1)$ is strongly connected.

    If $G$ preserves a symplectic or unitary form then all edges are two-way and we can omit $t'_k$.
\end{proof}

Recall that $T$ is weakly nondegenerate if $V(T)\cap V^*(T)^\perp = 0$ or $V(T)^\perp \cap V^*(T) = 0$.
The \emph{defect} of $T$ is
\[
    \delta(T) = \min (\dim V(T) \cap V^*(T)^\perp, \dim V(T)^\perp \cap V^*(T)).
\]

\begin{proposition}
    \label{prop:winkle}
    Let $T \subset \tv$ be a dense set of transvections.
    Let $T_0 \subset T$ be a nonempty subset such that $\Gamma(T_0)$ is strongly connected.
    Then there is a set of transvections $T_1 \subset T$ such that
    \begin{enumerate}
        \item $T_0 \subset T_1$,
        \item $|T_1| \le |T_0| + \delta(T_0)$,
        \item $\Gamma(T_1)$ is strongly connected,
        \item $T_1$ is weakly nondegenerate.
    \end{enumerate}
\end{proposition}
\begin{proof}
    Consider the pairing $V(T_0) \times V^*(T_0) \to K$ and assume both the left and right kernels are nonzero.
    Assume $0 \ne u \in V(T_0) \cap V^*(T_0)^\perp$ and $0 \ne \psi \in V(T_0)^\perp \cap V^*(T_0)$.
    By density there is some
    \[
        t_1 = 1 + v_1 \otimes \phi_1 \in T
    \]
    such that $\psi \to t_1 \to u$.
    Let $T_1 = T_0 \cup \{t_1\}$.
    Then $\Gamma(T_1)$ is again strongly connected.
    Moreover since $\psi \in V^*(T_1)$ and $V(T_1) \cap \psi^\perp = V(T_0)$ we have
    \[
        V(T_1) \cap V^*(T_1)^\perp = V(T_0) \cap V^*(T_1)^\perp = V(T_0) \cap V^*(T_0)^\perp \cap \phi_1^\perp.
    \]
    Since $u \in V(T_0) \cap V^*(T_0)^\perp \setminus \phi_1^\perp$, this implies that
    \[
        \dim V(T_1) \cap V^*(T_1)^\perp = \dim V(T_0) \cap V^*(T_0)^\perp - 1
    \]
    and by a symmetric argument we have
    \[
        \dim V(T_1)^\perp \cap V^*(T_1) = \dim V(T_0)^\perp \cap V^*(T_0) - 1.
    \]
    Thus $\delta(T_1) = \delta(T_0) - 1$
    and we are done by induction on defect.
\end{proof}
%
%
%

\section{Certification}

The main new idea in the present work is to leverage \Cref{thm:classification} using ``certificates''. By a \emph{certificate} for $T$ we mean a bounded-size subset $T_0 \subset \gen T$ such that $\ell_T(T_0) \le (n \log q)^{O(1)}$ and such that $G_0 = \gen {T_0}$ acts irreducibly on $V_0 = V(T_0) = [V, G_0]$ with the same defining field and same type as $G$ itself.
We may assume $\dim V_0 \ge 10$ to avoid pathologies.
As a consequence of Kantor's classification theorem, it follows in particular that if $T_0 \subset H \le G$ and $H$ is irreducible then $H = G$.

Consider the containment diagram shown in \Cref{fig:classification} between the types of irreducible groups generated by transvections with defining field $\F_q$.
Here we have ignored the groups of exceptional type. This is harmless as we can always choose $T_0$ so that $\dim V_0 \ge 10$.
Now our strategy is to consider each of the possibilities above in turn.
For example, if $G$ is not contained in a symplectic group, then we aim to find a subset $T_0$ of bounded size which is also not contained in a symplectic group: this is actually easy using symplectic cycles.
Some of the other types will give us much more difficulty.

\begin{figure}
    \[\begin{tikzpicture}[scale=3]
        \node (L) at (0,0) {$\SL_n(q)$};
        \node (S) at (-1,-1) {$\Sp_n(q)$};
        \node (U) at (+1,-1) {$\SU_n(q_0)$};
        \node (O) at (-1,-2) {$\Or^\pm_n(q)$};
        \node (Sym1) at (-1,-3) {$S_{n+1}$};
        \node (Sym2) at (-0,-3) {$S_{n+2}$};
        \node (M) at (1, -3) {$C_a^{n-1} \rtimes S_n$};

        \draw (L) -- node[above, rotate=45] {$2 \mid n$} (S);
        \draw (L) -- node[above, rotate=-45] {$q=q_0^2$} (U);
        \draw (S) -- node[left] {$p = 2$} (O);
        \draw (O) -- node[left] {$q=2$} (Sym1);
        \draw (O) -- node[above,rotate=-45] {$q=2, 4 \nmid n$} (Sym2);
        \draw (S) to[bend left=15] node[right] {$q=2$} (Sym2);
        \draw (L) -- (U);
        \draw (L) -- node[left] {$p = 2$} (M);
        \draw (U) -- node[above, rotate=-90] {$a \mid q_0 + 1$, $p = 2$} (M);
    \end{tikzpicture}\]
    \caption{Containment between types of irreducible groups generated by transvections with defining field $\F_q$}
    \label{fig:classification}
\end{figure}
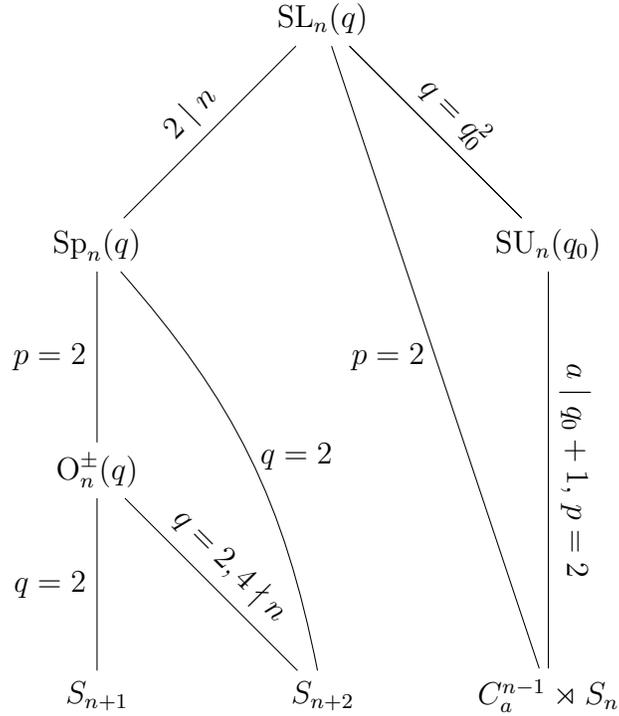

For the rest of the paper we assume $K = \F_q$ is a finite field of characteristic $p$. As in previous sections, $T \subset \tv$ and $G = \gen T$.
We assume $G$ is irreducible.
In consideration of \Cref{prop:density}, we assume $T$ is dense.

\subsection{Unitary and symplectic subgroups}

\begin{proposition}[cf.~\cite{GHS}*{Lemma~4.8}]
    \label{prop:unitary-symplectic-certificate}
    If $G$ does not preserve a nondegenerate symplectic (respectively, unitary) form then $\Gamma(T)$ contains a non-symplectic (respectively, non-unitary) cycle of length at most $5$.
\end{proposition}
\begin{proof}
    This is immediate from \Cref{prop:invariant-form}, since $\Gamma(T)$ has directed diameter at most $2$.
\end{proof}

\emph{Warning}:
The definition of unitary cycle depends on the field.
In particular even if $T_0 \subset T$ contains a cycle which is non-unitary with respect to $\F_q$, it is still quite possible that $G_0 = \gen{T_0}$ is a unitary group with respect to the subfield $L(T_0) < L(T)$.

\subsection{Symmetric-type subgroups}

Let us review the construction of the irreducible linear groups of symmetric type.
Let $G = S_m$ and consider its standard permutation representation $K^m$, where $K = \F_2$.
The unique proper nontrivial submodules are $\ell = (K^m)^G = \gen e$, where $e = (1, \dots, 1)$, and $H = [K^m, G] = \{x : x_1 + \cdots + x_m = 0\}$.
If $m$ is odd, $W = \ell \oplus H$.
If $m$ is even, $\ell \le H$.
Let $V = H / \ell \cap H$ and $n = \dim V = m - \gcd(m, 2)$.
The natural image of $G$ in $\SL(V)$ is a \emph{group of symmetric type}.
We call $G$ \emph{odd symmetric type} if $m$ is odd and \emph{even symmetric type} if $m$ is even.

The unique nonzero $G$-invariant alternating bilinear form on $K^m$ is given by $f(x, y) = \sum_{i \ne j} x_i y_j$,
and one checks $H^\perp = \ell$, so $f$ induces a $G$-invariant symplectic form on $V$.
Thus $G$ is conjugate to a subgroup of $\Sp_n(2)$.
One may also check that $G$ preserves a quadratic form on $V$ if and only if $m \not\equiv 2 \pmod 4$.

For small $m$ there are various coincidences with groups of classical type:
\begin{align*}
    &S_3 \cong S_4 / V \cong \SL_2(2), &&S_5 \cong \Or_4^-(2) \cong \SL_2(4) : 2, \\
    &S_6 \cong \Sp_4(2), &&S_8 \cong \Or_6^+(2) \cong \SL_4(2) : 2.
\end{align*}
It is therefore harmless to assume $m \ge 7$, $m \ne 8$.
Unless $m \in \{2,4\}$, an element of $G \cong S_m$ acts on $\{1, \dots, m\}$ as a transposition if and only if it acts on $V$ as a transvection.

\begin{lemma}
    \label{lem:transpositions-transvections}
    Let $G = S_m$ and let $\rho : G \to \SL(V) \cong \SL_n(2)$ be the representation defined above, where $n = m - \gcd(m, 2)$.
    Assume $m \ge 5$.
    An element $g \in G$ is a transposition if and only if $\rho(g)$ is transvection.
\end{lemma}
\begin{proof}
    Let $e_1, \dots, e_m$ be the standard basis of $K^m$. Define $e_i^* \in (K^m)^*$ by $e_i^*(e_j) = \delta_{ij}$.
    Let $g = (i,j) \in G$. Then $g$ acts as $1 + (e_i + e_j) \otimes (e_i^* + e_j^*)$ on $K^m$, which descends to a transvection on $V$, so $\rho(g)$ is a transvection.
    Conversely, suppose $\rho(g)$ is a transvection. Then $g$ is an involution, so $g$ has cycle type $2^k$ for some $k$,
    say $g = (1,k+1)(2,k+2)\cdots(k,2k)$ without loss of generality.
    Then $g$ acts as $1 + \sum_{i=1}^k (e_{i} + e_{k+i}) \otimes (e_{i}^* + e_{k+i}^*)$ on $K^m$.
    The subspace $(g-1)(H)$ consists of all vectors $x \in K^m$ such that $x_{i} = x_{k+i}$ for $1 \le i \le k$ and such that $x_i = 0$ for $i > 2k$,
    and moreover such that $x_1 + \cdots + x_k = 0$ if $k = m/2$,
    so $\dim (g-1)(H) = k$ if $k < m/2$ and $\dim (g-1)(H) = k-1$ if $k = m/2$.
    Also $(g-1)(H) \cap \ell = 0$ unless $k = m/2$ and $k$ is even.
    Thus $(g-1)|V$ has rank $k$ unless $k = m/2$, in which case it has rank $k-1$ if $k$ is odd and $k-2$ if $k$ is even.
    Therefore in order that $g$ acts a transvection on $V$ it is necessary that $k = 1$.
\end{proof}

In the rest of this subsection, we consider the following question. Let $K = \F_2$ and $V = K^n$, where $n$ is even. Given $T \subset \tv \cap \Sp_n(2)$ such that $G = \gen T$ is irreducible, how can we certify that $G$ is not contained in a group of symmetric type?

\begin{lemma}
    \label{lem:odd-sym-type}
    The following are equivalent:
    \begin{enumerate}
        \item $G$ has odd symmetric type,
        \item there is a $G$-invariant spanning set $B \subset V$ of cardinality $n+1$.
    \end{enumerate}
    Moreover $B$ is uniquely determined by $G$.
\end{lemma}
\begin{proof}
    \emph{(1) $\implies$ (2):} If $G = S_{n+1}$ acting on the hyperplane $H \le \F_2^{n+1}$ then $B = \left\{\sum_{i \ne j} e_i : j = 1, \dots, n+1 \right\}$ is a $G$-invariant spanning set of cardinality $n+1$.

    \emph{(2) $\implies$ (1):}
    Let $B = \{b_1, \dots, b_{n+1}\}$.
    Then $b_1 + \cdots + b_{n+1} \in V^G = 0$, since $G$ is irreducible.
    Since $|B| = n+1$, this is the unique nontrivial linear relation among the elements of $B$.
    In particular $G$ must be transitive on $B$, for if $B_0 \subset B$ is $G$-invariant then $\sum_{b \in B_0} b = 0$ by the same argument.

    Let $W = V \oplus \ell$ where $\ell = \{0,e\}$ is a copy of the trivial $1$-dimensional $G$-module $K^1$. Let $b'_i = b_i + e$ for $1 \le i \le n+1$ and let $\Omega = \{b'_1, \dots, b'_{n+1}\}$. Then $\Omega$ is $G$-invariant and
    \[
        b'_1 + \cdots + b'_{n+1} = b_1 + \cdots + b_{n+1} + (n+1) e = e.
    \]
    It follows that $\Omega$ spans $W$, so $\Omega$ is a $G$-invariant basis of $W$. Thus we may identify $W$ with $K^\Omega$, $V$ with the hyperplane $H \le K^\Omega$, and $G$ with a subgroup of $\Sym(\Omega)$.
    Hence $G$ is a transitive subgroup of $\Sym(\Omega)$ generated by transpositions (by \Cref{lem:transpositions-transvections}), so $G = \Sym(\Omega)$ by \Cref{prop:transposition-graph}.

    Now we argue that $B$ is uniquely determined.
    Assume $G$ has odd symmetric type and $B,B' \subset \F_2^n$ are two $G$-invariant spanning sets of cardinality $n+1$.
    The above argument shows that $G|B = \Sym(B)$ and $G|B' = \Sym(B')$.
    Since $\Aut(S_m) = S_m$ for $m \ne 2,6$ and $|B|$ is odd, there is a $G$-equivariant bijection $f : B \to B'$.
    Let $\tau \in G$ be a transposition, say transposing $b, b' \in B$. Then $\im(\tau -1) = \gen{b + b'}$.
    Also $\tau$ transposes $f(b), f(b')$, so $\im(\tau - 1) = \gen{f(b) + f(b')}$.
    It follows that $b+b' = f(b)+f(b')$ for any two $b, b' \in B$. Since $|B|$ is odd, we deduce that
    \[
        b = \sum_{b'} (b+b') = \sum_{b'} (f(b) + f(b')) = f(b)
    \]
    for all $b \in B$. Thus $B = B'$.
\end{proof}

The main result of this subsection is the following proposition.

\begin{proposition}
    \label{prop:sym-type-certificate}
    Assume $T \subset \tv \cap \Sp_n(2)$ is dense and $n \ge 2000$.
    The following are equivalent:
    \begin{enumerate}
        \item $G$ has symmetric type;
        \item for every subset $T_0 \subset T$ of cardinality $|T_0| \le 1000$ such that $G_0 = \gen{T_0}$ acts irreducibly on $V_0 = [V, G_0]$, the restriction $G_0 | V_0$ has odd symmetric type.
    \end{enumerate}
\end{proposition}
\begin{proof}
    \emph{(1) $\implies$ (2):}
    Let $T_0 \subset T$ be a set of cardinality $|T_0| \le 1000$ such that $G_0 = \gen {T_0}$ acts irreducibly on $V_0 = [V, G_0]$.
    We can identify $G$ with $S_m$ for $m \in \{n+1,n+2\}$ and $V$ with the section $H / (\ell \cap H)$ of the standard permutation module $K^m$.
    Then $T \subset \Sym(\Omega)$ is a set of transpositions.
    Let $\Gamma_0 = \Gamma(T_0)$ be the transposition graph defined by $T_0$.
    Then $G_0 = \gen{T_0}$ is the direct product of the full symmetric groups on each component of $\Gamma_0$, and $V_0$ is the image in $V$ of the space $W$ of all $x \in K^m$ such that $\sum_{i \in C} x_i = 0$ for each component $C$ of $\Gamma_0$.
    If $e \in W$ then every component of $\Gamma_0$ has even cardinality and $|T_0| \ge (n+2) / 2 > 1000$, contrary to hypothesis.
    Therefore $W \cap \ell = 0$ and $V_0 \cong W$ as $G_0$-modules.
    Since $G_0 | V_0$ is irreducible, it must be that $\Gamma_0$ has only one non-singleton component, and it must have odd order. Thus $G_0 | V_0$ has odd symmetric type.

    \def\A{\mathcal{A}}
    \emph{(2) $\implies$ (1):}
    Call $T_0 \subset T$ \emph{admissible} if $G_0 = \gen{T_0}$ acts irreducibly on $V_0 = [V, G_0]$.
    Let $\A_t$ be the collection of admissible subsets $T_0 \subset T$ of cardinality $|T_0| \le t$.
    By hypothesis, $G_0|V_0$ has odd symmetric type for every $T_0 \in \A_{1000}$.
    In this case by \Cref{lem:odd-sym-type} there is a unique $G_0$-invariant spanning set $B_0 \subset V_0$ of cardinality $\dim V_0 + 1$.
    Define
    \[
        \Delta(T_0) = \{b + b' : b, b' \in B_0, b \ne b'\} \qquad (T_0 \in \A_{1000}).
    \]
    Throughout this proof we will refer to the subgroup $G_0 \le G$, the subspace $V_0 \le V$, and the subset $B_0 \subset V_0$ constructed from admissible sets $T_0$ in this way.

    Suppose $T_0 \subset T_1$ where $T_0, T_1 \in \A_{1000}$.
    Correspondingly define $G_i$, $V_i$, $B_i$ for $i = 0,1$.
    Then $G_0 \le G_1$, $V_0 \le V_1$, and from our analysis of the implication \emph{(1) $\implies$ (2)} we know that $B_0$ is related to $B_1$ in the following rigid way: there is a $G_0$-equivariant injection $\iota : B_0 \to B_1$ such that
    \[
        b = \iota(b) + \sum_{b' \in B_1 \sm \iota(B_0)} b' \qquad\qquad (b \in B_0).
    \]
    Also $G_0$ acts trivially on $B_1 \sm \iota(B_0)$. In particular $b+b' = \iota(b) + \iota(b')$ for $b,b' \in B_0$, so $\Delta(T_0) \subset \Delta(T_1)$.

    Define
    \[
        \Delta = \bigcup_{T_0 \in \A_{400}} \Delta(T_0).
    \]
    We claim that, for every $v \in \Delta$, there is in fact some $T_1 \in \A_6$ such that $v \in \Delta(T_1)$.
    Suppose $v \in \Delta(T_0)$ for $T_0 \in \A_{400}$.
    Then $v = b_0 + b_0'$ for some $b_0, b_0' \in B_0$.
    By consideration of the transposition graph defined by $T_0 \subset G_0 \cong \Sym(B_0)$, there are $t, t' \in T_0$ that move $b_0, b_0'$, respectively.
    By \Cref{prop:winkle}, $\{t,t'\}$ is contained in some $T_1 \in \A_6$, and $T_0 \cup T_1$ is contained in some $T_2 \in \A_{808}$ (note that $|T_0 \cup T_1| \le 404$ and $\Gamma(T_0\cup T_1)$ is strongly connected).
    Consider the corresponding maps $\iota_0 : B_0 \to B_2$ and $\iota_1 : B_1 \to B_2$.
    Since $t$ acts as a transposition on each of $B_0, B_1, B_2$, it must be that $\iota_0(b_0) = b_2 = \iota_1(b_1)$ for some $b_2 \in B_2$ and $b_1 \in B_1$, and likewise $\iota_0(b'_0) = b'_2 = \iota_1(b'_1)$ for some $b'_2 \in B_2$ and $b'_1 \in B_1$, whence
    \[
        v = b_0 + b'_0 = b_2 + b_2' = b_1 + b_1' \in \Delta(T_1).
    \]
    This proves our claim.

    In particular it follows that $\Delta$ is $G$-invariant.
    Indeed, it suffices to prove that if $v \in \Delta$ and $t \in T$ then $t(v) \in \Delta$.
    By the previous paragraph, $v \in \Delta(T_0)$ for some $T_0 \in \A_6$.
    By \Cref{prop:connect-up,prop:winkle}, $T_0 \cup \{t\}$ is contained in some $T_1 \in \A_{16}$.
    Since $\Delta(T_1)$ is $G_1$-invariant and $t \in G_1$, $t(v) \in \Delta(T_1) \subset \Delta$.

    Similarly, $G$ is transitive on $\Delta$. Consider $v,v' \in \Delta$. Then $v \in \Delta(T_0)$ and $v' \in \Delta(T_0')$ for $T_0, T_0' \in \A_6$.
    By \Cref{prop:connect-up,prop:winkle}, $T_0 \cup T_0'$ is contained in some $T_1 \in \A_{26}$.
    We have $v, v' \in \Delta(T_0) \cup \Delta(T_0') \subset \Delta(T_1)$.
    Since $G_1 = \gen{T_1} \cong \Sym(B_1)$ is transitive on $\Delta(T_1)$, it follows that some element of $G_1$ maps $v$ to $v'$, and $G_1 \le G$.

    Next we define a symmetric relation $E$ on $\Delta$ by
    \[
        E = \{(u,v) \in \Delta^2 : u + v \in \Delta\}.
    \]
    Plainly $E$ is $G$-invariant, so $G$ acts on the graph $(\Delta, E)$ by automorphisms.
    Moreover we can describe the relation $E$ another way as follows.
    Suppose $(u,v) \in \Delta^2$ and suppose $u,v \in \Delta(T_0)$ for some $T_0 \in \A_{400}$.
    Then $u = b_1 + b_2$ and $v = b_3 + b_4$ for some $b_1,b_2,b_3,b_4 \in B_0$.
    We claim that $(u,v) \in E$ if and only if $|\{b_1,b_2\} \cap \{b_3,b_4\}| = 1$.
    Indeed, if say $b_1=b_3$ but $b_2 \ne b_4$ then $u + v = b_2 + b_4 \in \Delta(T_0) \subset \Delta$.
    Conversely, suppose $w = u + v \in \Delta$.
    Then $w \in \Delta(T_1)$ for some $T_1 \in \A_6$.
    By \Cref{prop:connect-up,prop:winkle} there is a set $T_2 \in \A_{814}$ containing $T_0 \cup T_1$.
    Since $u,v,w \in \Delta(T_0) \cup \Delta(T_1) \subset \Delta(T_2)$,
    each of $u,v,w$ is the sum of two distinct elements of $B_2$.
    Moreover the expressions for $u$ and $v$ are $u = \iota(b_1) + \iota(b_2)$ and $v = \iota(b_3) + \iota(b_4)$, where $\iota : B_0 \to B_2$ is the usual map.
    If $b_1,b_2,b_3,b_4$ are all distinct then $w = \iota(b_1) + \iota(b_2) + \iota(b_3) + \iota(b_4)$ is not the sum of two elements of $B_2$, since the unique nontrivial linear relation among the elements of $B_2$ is $\sum_{b \in B_2} b = 0$ and $|B_2|$ is odd.
    Therefore $|\{b_1,b_2\} \cap \{b_3,b_4\}| = 1$, as claimed.

    Let $\Omega$ be the set of maximal cliques $C$ of $(\Delta, E)$ of cardinality $|C| > 3$, and note that $\Omega$ inherits a $G$-action from $(\Delta, E)$.

    We claim that each clique $C \in \Omega$ is determined by any two of its elements, i.e., $|C \cap C'| \le 1$ for distinct $C, C' \in \Omega$.
    Suppose $C, C' \in \Omega$ share two vertices $v_1, v_2 \in \Delta$.
    We will show that $C \cup C'$ is a clique, whence $C = C' = C \cup C'$ by maximality.
    Let $v_3, v_4 \in C$ and $v_3', v'_4 \in C'$ be distinct from $v_1, v_2$.
    By \Cref{prop:connect-up,prop:winkle}, there is some $T_0 \in \A_{82}$ such that $v_1, v_2, v_3, v_4, v_3', v_4' \in \Delta(T_0)$.
    By our alternative characterization of $E$, each $v_i$ and $v'_i$ corresponds to a $2$-subset of $B_0$ in such a way that adjacent vertices correspond to overlapping subsets.
    Since $\{v_1,v_2,v_3,v_4\}$ is a clique, we must have $v_i = b_0 + b_i$ for $i = 1,2,3,4$ for some $b_0,b_1,b_2,b_3,b_4 \in B_0$,
    and similarly we must have $v'_3 = b_0 + b'_3$ and $b'_4 = b_0 + b'_4$.
    Thus $v_3 + v'_3 = b_3 + b'_3 \in \Delta \cup \{0\}$, and since $v_3$ and $v'_3$ were arbitrary this implies that $C \cup C'$ is a clique, as claimed.

    In fact we claim that $|C \cap C'| = 1$ for any two distinct cliques $C, C' \in \Omega$.
    Let $v_1, v_2, v_3, v_4 \in C$ be distinct and let $v'_1, v'_2, v'_3, v'_4 \in C'$ be distinct.
    By maximality and distinctness of $C$ and $C'$, we may assume that $v_1\ne v'_1$ and $(v_1,v'_1) \notin E$.
    By the usual arguments, $v_1,v_2,v_3,v_4,v'_1,v'_2,v'_3,v'_4 \in \Delta(T_0)$ for some $T_0 \in \A_{110}$.
    As above, $v_1, v_2,v_3,v_4$ correspond to $2$-subsets of $B_0$ overlapping in pairs, so we must have $v_i = b_0 + b_i$ for $i = 1,2,3,4$ and some $b_0,b_1,b_2,b_3,b_4 \in B_0$, and similarly $v'_i = b'_0 + b'_i$ for $i = 1,2,3,4$ for some $b'_0,b'_1,b'_2,b'_3,b'_4 \in B_0$.
    Let $w = b_0 + b'_0$, and observe that $w \ne 0$ since $v_1 \ne v'_1$ and $(v_1, v'_1) \notin E$.
    Then $\{v_1,v_2,v_3,v_4,w\}$ is a clique of cardinality at least $4$ containing $v_1,v_2$, so it must be contained in $C$.
    This shows $w \in C$ and symmetrically $w \in C'$, so $w \in C \cap C'$.

    We make one further observation. Let $C, C', C'' \in \Omega$ be distinct cliques.
    Let $w \in C \cap C'$ and $w' \in C' \cap C''$. Then
    \begin{equation}
        \label{eq:w+w'}
        w + w' \in C \cap C''.
    \end{equation}
    Indeed, let $v_i \in C, v'_i \in C', v''_i \in C''$ for $i=1,2,3,4$ as above.
    Find $T_0 \in \A_{166}$ such that $v_i,v'_i,v''_i \in \Delta(T_0)$ for $i=1,2,3,4$.
    Then we must have representations of the form $v_i = b_0 + b_i, v'_i = b'_0 + b'_i, v''_i = b''_0 + b''_i$, and $w = b_0 + b'_0$, $w' = b'_0 + b''_0$, and $w'' = b_0 + b''_0$.

    Finally, we claim that $G | \Omega = \Sym(\Omega)$. It suffices to prove that for any two distinct $C, C' \in \Omega$, there is an element of $G$ transposing $C$ and $C'$ and fixing every other element of $\Omega$.
    Exactly as above let $v_1,v_2,v_3,v_4 \in C$ be distinct,
    let $v'_1,v'_2,v'_3,v'_4 \in C'$ be distinct,
    and find $T_0 \in \A_{110}$ such that $v_i,v'_i \in \Delta(T_0)$ for $i=1,2,3,4$.
    We have $v_i = b_0 + b_i$ and $v'_i = b'_0 + b'_i$ ($i=1,2,3,4$) for some $b_i,b'_i \in B_0$ ($i = 0,1,2,3,4$).
    Since $G_0 \cong \Sym(B_0)$, there is some $t \in G_0$ transposing $b_0$ and $b'_0$ and fixing every other element of $B_0$.
    Observe that $\{v_1,v_2,v_3,v_4, t(v'_1), t(v'_2), t(v'_3), t(v'_4)\}$ is a clique.
    Since $C$ is the unique maximal clique containing $v_1,v_2,v_3,v_4$ and $t(C')$ is the unique maximal clique containing $v'_1,v'_2,v'_3,v'_4$, it follows that $t(C') = C$, and similarly $t(C) = C'$.
    Now consider any other $C'' \in \Omega$ and let $v''_1,v''_2,v''_3,v''_4 \in C''$.
    We can find $T_1 \in \A_{276}$ containing $T_0$ such that $v''_1,v''_2,v''_3,v''_4 \in \Delta(T_1)$,
    and as above we must have $v''_i = b''_0 + b''_i$ ($i=1,2,3,4$) for some $b''_i \in B_1$ ($i=0,1,2,3,4$).
    If say $b''_0 = \iota(b_0)$ then $\{v_1,v_2,v''_1,v''_2\}$ is a clique, which implies $C'' = C$.
    Likewise $b''_0 = \iota(b'_0)$ implies $C'' = C'$.
    Otherwise $b''_0$ is fixed by $t$, so $\{v''_1,v''_2,t(v''_1),t(v''_2)\}$ is a clique, which implies $C'' = t(C'')$, and this is what we wanted to prove.

    At last we can complete the proof. Consider the $G$-space $K^\Omega$ and let $H \le K^\Omega$ be the hyperplane $\{x \in K^\Omega : \sum_{C \in \Omega} x_C = 0\}$.
    Fix any $C_0 \in \Omega$ and observe that $\{e_{C_0} + e_C : C \ne C_0\}$ is a basis of $H$.
    Define $f : H \to V$ by sending $e_{C_0} + e_C$ to the unique element of $C_0 \cap C$ (recalling that the elements of $\Omega$ are by definition cliques of $(\Delta, E$)) and linearly extending to $H$.
    It follows from \eqref{eq:w+w'} that $f$ sends $e_{C} + e_{C'}$ to the unique element of $C \cap C'$ for all $C, C' \in \Omega$, $C \ne C'$.
    Hence $f$ is a nonzero $G$-module homomorphism, and it must be surjective since $V$ is irreducible.
    Thus $G$ has symmetric type, as claimed.
    \end{proof}

    \begin{remark}
    In the above proof, it follows a posteriori that $|\Omega| \in \{n+1,n+2\}$, and $G$ has odd type or even type according to these two cases.
    \end{remark}

\subsection{Monomial subgroups}

Let $K = \F_q$ be a finite field of characteristic $2$ and let $a$ be a divisor of $q-1$. Let $M_n(a)$ be the subgroup of $\SL_n(q)$ consisting of those elements which are represented in the standard basis by monomial matrices with all nonzero entries being $a$th roots of unity. Then $M_n(a) \cong C_a^{n-1} \rtimes S_n$ and $M_n(a)$ is a subgroup of $\SL_n(q)$ generated by transvections, irreducible for $a > 1$. The transvections in $M_n(a)$ are the elements represented as
\[
    \begin{pmatrix}
    0 & x^{-1} & 0 \\
    x & 0 & 0 \\
    0 & 0 & I_{n-2}
    \end{pmatrix}
\]
up to permuting axes.

In general if $G \le \GL(V)$, a \emph{$G$-invariant monomial structure} on $V$ is a set of lines $L = \{Kb_1, \dots, Kb_n\}$ permuted by $G$ such that $b_1, \dots, b_n$ is a basis for $V$.

\begin{lemma}
    If $a > 1$ there is a unique $M_n(a)$-invariant monomial structure on $V = K^n$.
\end{lemma}
\begin{proof}
    First observe that if $N \nsgp C_a^{n-1} \rtimes S_n$ and $N$ is not contained in $C_a^{n-1}$ then $N$ contains $C_a^{n-1}$.
    Indeed, $N$ contains some element with nontrivial image $\pi$ in $S_n$, so $N$ contains $[C_a^{n-1}, \pi]$.
    Now it follows by a short argument that $N$ must contain all of $C_a^{n-1}$.
    In particular if $N$ is normal and abelian then $N \le C_a^{n-1}$.

    Now let $G = M_n(a) \cong C_a^{n-1} \rtimes S_n$.
    The fact that $\{Ke_1, \dots, Ke_n\}$ is a $G$-invariant monomial structure is obvious from the definition of $G$.
    Suppose $L = \{Kb_1, \dots, Kb_n\}$ is another $G$-invariant monomial structure.
    The kernel $C_G(L)$ of the action of $G$ on $L$ is abelian and normal, so by the above argument it must be contained in the diagonal subgroup $D \cong C_a^{n-1}$ of $G$.
    Moreover $G/C_G(L)$ is isomorphic to a subgroup of $\Sym(L) \cong S_n$; it follows that $C_G(L) = D$.
    Hence $b_1, \dots, b_n$ are joint eigenvectors of $D$.
    However it is straightforward to check that the only joint eigenvectors of $D$ are the scalar multiples of $e_1, \dots, e_n$, so $\{Kb_1, \dots, Kb_n\} = \{Ke_1, \dots, Ke_n\}$.
\end{proof}

Note that $M_n(a)$ does \emph{not} preserve a symplectic form on $V$ if $a>1$.
Indeed, the only bilinear forms invariant on the diagonal subgroup of $M_n(a)$ are those proportional to $\sum_{i=1}^n x_i y_i$, and the only such form that is alternating is the zero form.
By similar reasoning $M_n(a)$ preserves a unitary form if and only if $a \mid q+1$.

\begin{proposition}
    \label{prop:monomial-type-certificate}
    Let $T \subset \tv$ be a dense set of transvections in $\SL_n(q)$.
    Assume that $G$ does not preserve a symplectic form.
    Then the following are equivalent:
    \begin{enumerate}
        \item $G$ has monomial type;
        \item for every subset $T_0 \subset T$ of cardinality $|T_0| \le 1000$,
        \begin{enumerate}[(a)]
            \item if $T_0$ is weakly nondegenerate then it is nondegenerate, and
            \item if $G_0 = \gen{T_0}$ acts irreducibly on $V_0 = [V, G_0]$ then $G_0 | V_0$ has monomial type or odd symmetric type.
        \end{enumerate}
    \end{enumerate}
\end{proposition}
\begin{proof}
    \emph{(1) $\implies$ (2):}
    Assume $G = M_n(a) \le \SL_n(q)$ and recall $M_n(a) \cong C_a^{n-1} \rtimes S_n$.
    Let $T_0 \subset T$ be a subset (of any cardinality).
    Observe that the image of $T_0$ under the natural map $M_n(a) \to S_n$ is a set of transpositions.
    Let $\Gamma_0$ be the corresponding transposition graph.
    Then by \Cref{prop:transposition-graph} the image of $G_0$ in $S_n$ is the direct product of the full symmetric groups on each component $C$ of $\Gamma_0$.
    The group $G_0$ itself is the direct product of monomial groups $M_d(b)$ for $b \ge 1$,
    say $G_0 = \prod_{i=1}^k M_{d_i}(b_i)$, where $\sum d_i = n$, $d_i \ge 1$, and we take $b_i = 1$ if $d_i = 1$.
    Correspondingly $[V, G_0]$ is the direct sum of subspaces $U_i \le K^{d_i}$, of codimension $1$ if $b_i = 1$ and codimension zero if $b_i > 1$,
    while $V^{G_0}$ is the direct sum of subspaces $W_i \cong K$ if $b_i = 1$ and $W_i = 0$ if $b_i > 1$.
    Therefore $\dim V(T_0) = \dim [V,G_0] = n - \ell$ where $\ell$ is the number of $b_i = 1$
    and also $\dim V^*(T_0) = n - \dim V^{G_0} = n - \ell$.
    Therefore $\dim V(T_0) = \dim V^*(T_0)$, so weak nondegeneracy implies nondegeneracy.

    Now assume $G_0|V_0$ is irreducible. Since $V_0 = \bigoplus_{i: d_i > 1} U_i$ and the $U_i$ are $G_0$-invariant, there must be exactly one $d_i > 1$, say $d_1 > 1$ (i.e., $\Gamma_0$ is connected apart from isolated vertices). Moreover $G_0|V_0$ is monomial type or symmetric type $1$ according to whether $b_1 > 1$ or $b_1 = 1$, and in the latter case $d_1$ must be odd.

    \def\A{\mathcal{A}}
    \emph{(2) $\implies$ (1):}
    Since $G$ does not preserve a symplectic form, by \Cref{prop:unitary-symplectic-certificate} there is a set $T_0$ of cardinality $|T_0| \le 5$ inducing a non-symplectic cycle.
    Let $\A_t$ be the set of subsets $T_1 \subset T$ of cardinality $|T_1| \le t$ containing $T_0$ such that $\Gamma(T_1)$ is strongly connected and $T_1$ is weakly nondegenerate.
    By hypothesis, every $T_1 \in \A_{1000}$ is nondegenerate. Thus $G_1 = \gen{T_1}$ acts irreducibly on $V_1 = [V, G_1]$ and we have a direct sum decomposition $V = V_1 \oplus V^{G_1}$.
    Moreover, since $T_0 \subset G_1$, $G_1$ does not have symmetric type, so it must have monomial type.
    By the previous lemma, there is a unique $G_1$-invariant monomial structure $L_1$ on $V_1$.
    Let $L$ be the union of the sets $L_1$ over all $T_1 \in \A_{100}$.
    Note that $L \ne \emptyset$ since \Cref{prop:winkle} implies $T_0$ is contained in some $T_1 \in \A_{10}$.

    Let $t \in T$. We claim that $t$ transposes some two lines $Kb, Kb' \in L$ and acts trivially on every other line in $L$. By \Cref{prop:connect-up,prop:winkle}, $T_0 \cup \{t\}$ is contained in some $T_1 \in \A_{16}$. The group $G_1$ has monomial type, and therefore the transvections in $G_1$ act as transpositions on $L_1$. Thus $t \in G_1$ transposes some two lines $Kb, Kb' \in L_1$.
    Now let $Kb'' \in L$ be any other line. Then $Kb'' \in L_2$ for some $T_2 \in \A_{100}$, and by \Cref{prop:connect-up,prop:winkle} the set $T_1 \cup T_2$ is contained in some $T_3 \in \A_{236}$.
    Now $G_3|V_3$ is monomial type with unique monomial structure $L_3$ such that $L_1 \cup L_2 \subset L_3$.
    By the structure of transvections in monomial groups, $t$ acts trivially on each line in $L_3 \sm \{Kb,Kb'\}$, including $Kb''$.

    In particular, $L$ is $G$-invariant. To complete the proof it suffices to show that the lines of $L$ are linearly independent, i.e., if $L = \{Kb_1, \dots, Kb_m\}$ then $b_1, \dots, b_m$ are linearly independent.
    Indeed, since $L$ is $G$-invariant it follows by irreducibility that $m=n$ and $L$ is a $G$-invariant monomial structure.
    This shows that $G$ is a subgroup of a group of monomial type, and, by the argument we used to show \emph{(1) $\implies$ (2)}, any irreducible subgroup generated by transvections of a group of monomial type not preserving a symplectic form itself must be of monomial type.

    Suppose
    \[
        \lambda_1 b_1 + \cdots + \lambda_m b_m = 0 \qquad (\lambda_1, \dots, \lambda_m \in K).
    \]
    Say $Kb_1 \in L_1$ for some $T_1 \in \A_{100}$.
    Since $G_1$ has monomial type, $G_1$ contains some element $g$ mapping $b_1 \mapsto xb_1$ ($x \ne 1$) and say $b_2 \mapsto x^{-1}b_2$, and acting trivially on every other line of $L_1$.
    Consider any other line $Kb \in L$. We have $Kb \in L_2$ for some $T_2 \in \A_{100}$ and by \Cref{prop:connect-up,prop:winkle} we have $T_1 \cup T_2 \subset T_3$ for some $T_3 \in \A_{404}$. Since $G_1$ is a monomial-type subgroup of the monomial-type group $G_3$, the element $g$ acts trivially on $Kb$.
    Thus $g(b_1) = xb_1$, $g(b_2) = x^{-1}(b_2)$, and $g(b_i) = b_i$ for $i > 2$. Applying $g-1$, we thus have
    \[
        \lambda_1 (x-1) b_1 + \lambda_2 (x^{-1} -1) b_2 = 0.
    \]
    Since $Kb_1 \cap Kb_2 = 0$, this implies $\lambda_1 = \lambda_2 = 0$. Since this argument works for each index, it follows that $b_1, \dots, b_m$ are linearly independent, as required.
\end{proof}

Summarizing this subsection and the previous one, if $T \subset \tv$ is dense and $G = \gen T$ is classical, we can find a bounded-size certificate $T_0 \subset T$ for the classicality of $G$. This is the content of the next proposition.

\begin{lemma}
    [\cite{GHS}*{Lemma~4.3}]
    \label{lem:L5}
    If $T$ is dense then $L_5(T) = L(T)$.
\end{lemma}

\begin{proposition}
    \label{prop:classical-certificate}
    Let $n$ be sufficiently large.
    Let $T \subset \tv$ be dense and assume $G = \gen T$ has classical type and defining field $L(T) = \F_q$.
    Then there is a subset $T_0 \subset T$ of cardinality $|T_0| \le 10^5$ such that $G_0 = \gen{T_0}$ acts irreducibly on the section $W = [V, G_0] / [V, G_0]^{G_0}$ with classical type (over some subfield of $\F_q$) and $\dim W > 10$.
\end{proposition}
\begin{proof}
    By simply taking a union of suitable sets we can find a set $T_1 \subset T$ of cardinality $|T_1| \le 10^4$ with the following properties:
    \begin{enumerate}
        \item $\dim V(T_1) \ge 2000$;
        \item if $q \ne 2$, $\Gamma(T_1)$ contains a cycle of length at most $5$ whose weight is not contained in $\F_2$ (\Cref{lem:L5});
        \item if $G$ does not preserve a symplectic form then $\Gamma(T_1)$ contains a non-symplectic cycle of length at most $5$ (\Cref{prop:unitary-symplectic-certificate});
        \item if $q = 2$ and $G$ preserves a symplectic form then $T_1$ has a subset $T_2$ of cardinality $|T_2| \le 1000$ such that $G_2 = \gen{T_2}$ acts irreducibly on $V_2 = [V,G_2]$ and the restriction $G_2 | V_2$ does not have odd symmetric type (\Cref{prop:sym-type-certificate});
        \item if $G$ does not preserve a symplectic form and $p=2$ then there is a subset $T_2 \subset T_1$ such that either
        \begin{enumerate}[(a)]
            \item $T_2$ is weakly nondegenerate but degenerate, or
            \item $G_2 = \gen{T_2}$ acts irreducibly on $V_2 = [V, G_2]$ with neither monomial type nor odd symmetric type
        \end{enumerate}
        (\Cref{prop:monomial-type-certificate}).
    \end{enumerate}
    By density of $T$, there is a set $T_1' \subset T$ of cardinality $|T_1'| \le 2|T_1|$ such that $V(T_1) \cap V^*(T_1')^\perp = 0$ and $V^*(T_1) \cap V(T_1')^\perp = 0$.
    Finally, by \Cref{prop:connect-up}, $T_1 \cup T_1'$ is contained in a set $T_0 \subset T$ of cardinality $|T_0| \le 6|T_1| \le 10^5$ such that $\Gamma(T_0)$ is strongly connected.

    Now consider the image $G_0 | W$ of $G_0$ in $\GL(W)$, where $W$ is the section $[V,G_0] / [V,G_0]^{G_0}$.
    By \Cref{prop:irreducible-section}, $G_0 | W$ is irreducible and there is natural surjective graph homomorphism $\Gamma(T_0) \to \Gamma(T_0 | W)$ respecting the weights of cycles.
    We claim that $T_0 | W$ inherits the properties listed above.
    Apart from (5)(a), this is clear.
    In the case of (5)(a), the argument is as follows.
    Note that, since $V(T_1) \cap V^*(T_0)^\perp \le V(T_1) \cap V^*(T_1')^\perp = 0$, the subspace $V(T_1) \le V(T_0)$ maps injectively into $W$.
    Similarly, $W^*$ is naturally identified with $V^*(T_0) / V^*(T_0) \cap V(T_0)^\perp$, and since $V^*(T_1) \cap V(T_0)^\perp = 0$ it follows that $V^*(T_1)$ maps injectively into $W^*$.
    Thus if $T_2 \subset T_1$ is weakly nondegenerate but degenerate then so is its restriction $T_2 | W$.

    It follows from \Cref{prop:sym-type-certificate} and \Cref{prop:monomial-type-certificate} that $G_0|W$ has neither symmetric type nor monomial type, and $\dim W > 6$, so $G_0 | W$ must have classical type.
\end{proof}

\subsection{Subfield subgroups}

Now let $K = \F_q$ be an arbitrary finite field, $T \subset \tv$, $G = \gen T \le \SL_n(q)$.
Assume $T$ is dense, $G$ is classical, and $K = L(T)$.
We aim to find a certificate $T_0 \subset T^{(\log q)^{O(1)}}$ of cardinality $|T_0| \le 3$ such that $L(T_0) = K$.

\begin{lemma}
    Let $G \le \SL_n(q)$ be a classical group generated by transvections with defining field $\F_q$.
    Then there is a subset $T_0 \subset G \cap \tv$ of cardinality $|T_0| \le 3$ such that $L(T_0) = \F_q$.
\end{lemma}
\begin{proof}
    Let $\lambda \in \F_q$ be a primitive element.

    If $G = \SL_2(q)$ then we can take $T_0 = \{t,s\}$ for $t = 1 + \lambda e_1 \otimes e_2^*$ and $s = 1 + e_2 \otimes e_1^*$.
    Here $e_i^*$ is defined by $e_i^*(e_j) = \delta_{ij}$.
    Observe that
    \[
        w(t,s) = \tr((t-1)(s-1)) = \tr(\lambda e_1 \otimes e_1^*) = \lambda.
    \]
    Therefore $L(T_0)$ contains $\F_p(\lambda) = \F_q$.
    This completes the proof for linear and symplectic groups since both contain copies of $\SL_2(q)$.

    Now consider $G = \SU_3(\sqrt q)$. Since all unitary forms are equivalent it suffices to consider the unitary form
    \[
        f(x, y) = x_1 y_2^\theta + x_2 y_1^\theta + x_3 y_3^\theta,
    \]
    where $x^\theta = x^{\sqrt q}$ is the involutary automorphism.
    Now for $v \in V = \F_q^3$ let us denote by $v^* \in V^*$ the map $u \mapsto f(u, v)$.
    Note that $e_1, e_2$ is a hyperbolic pair.
    Let $\eps \in \F_q$ be such that $\eps^\theta = - \eps$ and consider $t_i = 1 + \eps v_i \otimes v_i^*$ for $i=1,2,3$, where $v_1 = e_1$, $v_2 = e_2$, and
    \[
        v_3 = \eps^{-3} \lambda e_1 + e_2 + z e_3,
    \]
    where $z$ is chosen so that $v_3$ is singular (recall that $\Fix(\theta) = \{zz^\theta : z \in K\}$ for finite fields).
    Since $v_1,v_2,v_3$ are singular, we have $t_1, t_2, t_3 \in G$, and
    \[
        w(t_1,t_2,t_3) = \eps^3 f(v_2, v_1) f(v_3, v_2) f(v_1, v_3)
        = \lambda.
    \]
    Since $\SU_3(\sqrt q) \le \SU_n(\sqrt q)$ for $n \ge 3$ this completes the proof for unitary groups.

    Finally consider an orthogonal group $G = \Or_n^\pm(q)$, $n \ge 6$, $p=2$.
    The transvections in $G$ are precisely those of the form $1 + v \otimes v^*$ where $Q(v) = 1$
    and $v^*$ is defined by $v^*(u) = f(u, v)$.
    Observe that
    \[
        w(1 + u\otimes u^*, 1 + v \otimes v^*) = f(u,v)^2.
    \]
    Let $H \le V$ be a $4$-dimensional hyperbolic subspace (see \cite{aschbacher-book}*{(21.2)}) and let $e_1,e_2,e_3,e_4$ be a hyperbolic basis for $H$.
    With respect to this basis we have, for $x, y \in H$,
    \begin{align*}
        Q(x) &= x_1 x_2 + x_3 x_4, \\
        f(x, y) &= x_1 y_2 + x_2 y_1 + x_3 y_4 + x_4 y_3.
    \end{align*}
    Let $u = e_1 + e_2$ and $v = \lambda e_1 +  e_3 + e_4$. Then $Q(u) = Q(v) = 1$ and $f(u,v) = \lambda$. Since finite fields are perfect, $\lambda^2$ generates $\F_q$, so it suffices to take $T_0 = \{1 + u \otimes u^*, 1 + v \otimes v^*\}$.
\end{proof}

\begin{proposition}
    \label{prop:subfield}
    Let $T \subset \tv$ be dense, $G = \gen T$ classical, $L(T) = \F_q$, and assume $n$ is sufficiently large.
    Then there is a subset $S \subset T^{(\log q)^{O(1)}} \cap \tv$ of cardinality $|S| \le 3$ such that $L(S) = \F_q$.
\end{proposition}
\begin{proof}
    First assume $G \ne \SL_n(q)$.
    Then there is a $G$-invariant symplectic or unitary form $f$.
    The presence of this form simplifies the argument enough that it is worth keeping it separate from the $\SL_n(q)$ case.

    First, by \Cref{prop:classical-certificate} and \Cref{prop:winkle}, there is a nondegenerate subset $T_0 \subset T$ of bounded cardinality such that $G_0 = \gen {T_0}$ acts irreducibly on $U_0 = [V, G_0]$ with quasisimple classical type.
    Let $K_0 = L(T_0)$.
    Now for $i > 0$ proceed as follows.
    Assume $T_{i-1} \subset \tv$ is a nondegenerate set of transvections of bounded cardinality such that $G_{i-1} = \gen {T_{i-1}}$ acts on $U_{i-1} = [V, G_{i-1}]$ with quasisimple classical type and defining field $K_{i-1} = L(T_{i-1})$.
    If $K_{i-1} = \F_q$, stop.
    Otherwise, let $H_i \le G_{i-1}$ be a quasisimple classical subgroup with the same defining field as $G_{i-1}$ and of rank at most $6$ (e.g., a copy of $\SL_2$, $\SU_3$, or $\Or_6^\pm$ over $K_{i-1}$).
    By the previous lemma, there is a subset $T_i^0 \subset H_i \cap \tv$ of cardinality at most $10$ generating $H_i$.
    By \Cref{lem:L5}, there is a cycle $T_i^1$ in $\Gamma(T)$ of length at most $5$ whose weight is not contained in $K_{i-1}$.
    Applying \Cref{prop:connect-up,prop:winkle}, $T_i^0 \cup T_i^1$ is contained in a nondegenerate set $T_i \subset T \cup (G_{i-1} \cap \tv)$ of cardinality $|T_i| \le 34$ such that $\Gamma(T_i)$ is strongly connected.
    Let $G_i = \gen{T_i}$ and $K_i = L(T_i)$.
    Let $U_i = [V, G_i]$. Then $V = U_i \oplus U_i^\perp$ since $T_i$ is nondegenerate, and $G_i$ acts irreducibly on $U_i = [V, G_i]$ by \Cref{prop:irreducible-section}.
    Since $H_i \le G_i$ and $H_i$ has quasisimple classical type, it follows from \Cref{thm:classification} that $G_i$ also has quasisimple classical type on $U_i$.
    Moreover, $G_i$ has defining field $K_i > K_{i-1}$.

    The process ends with some $K_r = \F_q$. Then $K_0 < \cdots < K_r = \F_q$. For each $i$, by \Cref{main-thm-bounded-rank}, we have
    \[
        \ell_{T_i \cup H_i}(G_i) \le O\pfrac{\log |G_i|}{\log |H_i|}^{O(1)} = O\pfrac{\log |K_i|}{\log |K_{i-1}|}^{O(1)}.
    \]
    Also $\ell_{T_0}(G_0) \le O(\log |K_0|)^{O(1)}$.
    Therefore, since $H_i \le G_{i-1}$ for each $i$,
    \begin{align*}
        \ell_T(G_r) \le \ell_T(G_0) \prod_{i=1}^r \ell_{T \cup H_i}(G_i)
        &\le O(\log |K_0|)^{O(1)} \prod_{i=1}^r O\pfrac{\log |K_i|}{\log |K_{i-1}|}^{O(1)}\\
        &\le O(1)^r \br{\log|K_r|}^{O(1)}.
    \end{align*}
    Now since $|K_i| \ge |K_{i-1}|^2$ for each $i$ we have $q \ge 2^{2^r}$, so it follows that
    \[
        \ell_T(G_r) \le O(\log q)^{O(1)}.
    \]
    Finally, by the previous lemma there is a subset $S \subset G_r \cap \tv$ of cardinality $|S| \le 3$ such that $L(S) = \F_q$, as required.

    Now suppose $G = \SL_n(q)$. We argue similarly but because we cannot guarantee nondegeneracy we instead work with the section $W_i = [V, G_i] / [V, G_i]^{G_i}$ throughout the argument.
    First we define a bounded-cardinality starting set $T_0 \subset T$.
    Using \Cref{prop:unitary-symplectic-certificate} we may include in $T_0$ a non-symplectic cycle of length at most $5$.
    If $q$ is a square, we also include a non-unitary cycle, as well as a cycle whose weight is not in $\F_{\sqrt q}$, using \Cref{lem:L5}.
    Using \Cref{prop:classical-certificate,prop:connect-up,prop:winkle}, we may assume $T_0$ is weakly nondegenerate, $\Gamma(T_0)$ is strongly connected, and $G_0 = \gen{T_0}$ acts irreducibly and classically on $W_0 = [V, G_0] / [V, G_0]^{G_0}$, and $\dim W_0 > 10$.
    Because of the presence of non-symplectic and non-unitary cycles, $G_0 | W_0 \cong \SL_{\dim W_0}(K_0)$ where $K_0 = L(T_0)$.
    Now for $i > 0$ proceed as follows.
    Assume $T_{i-1} \subset \tv$ is a weakly nondegenerate set of transvections of bounded cardinality such that $G_{i-1} = \gen{T_{i-1}}$ induces $\SL_{\dim W_{i-1}}(K_{i-1})$ on $W_{i-1} = [V, G_{i-1}] / [V, G_{i-1}]^{G_{i-1}}$, and $\dim W_{i-1} > 10$.
    If $K_{i-1} = \F_q$, stop.
    Otherwise, by the previous lemma and \Cref{prop:tv-in-SL-sections} there is a subset $T_i^0 \subset G_{i-1} \cap \tv$ of cardinality at most $3$ (actually $2$) such that $L(T_i^0) = K_{i-1}$.
    By \Cref{lem:L5}, there is a cycle $T_i^1$ in $\Gamma(T)$ of length at most $5$ whose weight is not contained in $K_{i-1}$.
    Let $t = 1 + v \otimes \phi \in T_i^0$ and let $H_i = 1 + K_{i-1} v \otimes \phi \cong K_{i-1}$ be the full transvection subgroup over $K_{i-1}$ generated by $t$.
    It follows from \Cref{prop:tv-in-SL-sections} that $H_i \le G_{i-1}$.
    Applying \Cref{prop:connect-up,prop:winkle}, $T_0 \cup T_i^0\cup T_i^1$ is contained in a bounded-cardinality weakly nondegenerate set $T_i \subset T \cup (G_{i-1} \cap \tv)$ such that $\Gamma(T_i)$ is strongly connected.
    Let $G_i = \gen{T_i}$ and $W_i = [V, G_i] / [V, G_i]^{G_i}$ and $K_i = L(T_i)$. Then $K_i > K_{i-1}$ and, since $T_0 \subset T_i$, $G_i | W_i \cong \SL_{\dim W_i}(K_i)$.
    Since $W_0$ is a section of $W_i$, $\dim W_i \ge \dim W_0 > 10$.

    As before the process ends with some $K_r = \F_q$. Then $K_0 < \cdots < K_r = \F_q$.
    Now we use \Cref{main-thm-bounded-rank} in combination with the diameter bound in \Cref{prop:tv-in-SL-sections}.
    For each $i$ we have
    \[
        \ell_{T_i \cup H_i}(G_i) \le O\pfrac{\log |G_i|}{\log |(H_i|W_i)|}^{O(1)} = O\pfrac{\log|K_i|}{\log|K_{i-1}|}^{O(1)}.
    \]
    Also $\ell_{T_0}(G_0) \le O(\log |K_0|)^{O(1)}$.
    Since $H_i \le G_{i-1}$ for each $i$ we get $\ell_T(G_r) \le O(\log q)^{O(1)}$ as before.
    Finally we apply the previous lemma in combination with \Cref{prop:tv-in-SL-sections} again to find $S \subset G_r \cap \tv$ as required.
\end{proof}

\subsection{Orthogonal subgroups}

The last thing remaining is distinguish orthogonal groups from symplectic groups. Throughout this subsection assume $K = \F_q$ is a finite field of characteristic $2$.

\begin{lemma}
    \label{lem:Q-surjective-on-affine-subspace}
    Let $(V, Q)$ be a nondegenerate quadratic space over a finite field $K = \F_q$.
    Let $w + H \subset V$ be an affine subspace of codimension $2$.
    Assume $\dim V \ge 10$.
    Then $Q(w + H) = K$.
\end{lemma}
\begin{proof}
    For $v \in H \cap w^\perp$ we have $Q(w + v) = Q(w) + Q(v)$, so it suffices to show that $Q(H_1) = K$ for $H_1 = H \cap w^\perp$. Write $H_1 = (H_1 \cap H_1^\perp) \oplus H_2$. Then $H_2$ is nondegenerate and $H_2$ has codimension at most $6$, so $\dim H_2 \ge 4$.
    It follows from the classification of quadratic forms over finite fields
    \cite{aschbacher-book}*{(21.2)} that $Q(H_2) = K$.
\end{proof}

Now assume $T \subset \tv \cap \Sp_n(q)$ is dense, $L(T) = K = \F_q$, and $G = \gen T$ is classical.

\begin{proposition}
    \label{prop:orthogonal-certificate}
    There is a constant $C$ such that the following are equivalent:
    \begin{enumerate}
        \item $G$ is an orthogonal group,
        \item every nondegenerate subset $T_0 \subset T^{(n \log q)^{O(1)}}$ of cardinality at most $C$ generates a subgroup $G_0 = \gen{T_0}$ that preserves an orthogonal structure on $V_0 = V(T_0) = [V, G_0]$.
    \end{enumerate}
\end{proposition}
\begin{proof}
    \emph{(1) $\implies$ (2):} Suppose $G$ preserves an orthogonal form $Q$ with corresponding symplectic form $f$. If $T_0 \subset T$ is nondegenerate then $V_0 = V(T_0)$ is an $f$-nondegenerate subspace of $V$. Thus $Q_0 = Q|V_0$ is a $G_0$-invariant orthogonal form on $V_0$.

    \emph{(2) $\implies$ (1):}
    First, by \Cref{prop:classical-certificate,prop:subfield,prop:connect-up,prop:winkle}, there is a bounded-cardinality nondegenerate subset $T_0 \subset T^{(\log q)^{O(1)}}$ such that $G_0 = \gen{T_0}$ acts irreducibly and classically on $V_0 = V(T_0)$ , $\dim V_0 \ge 10$, and $L(T_0) = K$. Note $G_0$ must be an orthogonal group. By \Cref{main-thm-bounded-rank},
    \[
        \ell_T(G_0) \le O(\log q)^{O(1)}, \qquad \ell_{G_0 \cap \tv}(G_0) \le O(1).
    \]
    Thus by replacing $T$ with $T^{(\log q)^{O(1)}} \cap \tv$ we may assume $T_0 \subset T$ and $\ell_T(G_0) \le O(1)$.

    Now we apply \Cref{prop:orthogonal}. For each $t \in T$ there is a unique $v_t \in V$ such that $t = 1 + v_t \otimes v_t^*$.
    Consider some linear relation
    \[
        \sum_{t \in T} \lambda_t v_t = 0.
    \]
    Suppose $\lambda$ is supported on $t_1, \dots, t_m$, and we may assume $m \le n+1$.
    We must show that $\tilde Q(\lambda) = 0$, where
    \[
        \tilde Q(\lambda) = \sum_{t \in T} \lambda_t^2 + \sum_{\{t,s\} \in \binom{T}{2}} \lambda_s \lambda_t f(v_s, v_t).
    \]
    Define also
    \[
        \tilde f (\lambda, \lambda') = \tilde Q(\lambda+\lambda') - \tilde Q(\lambda) - \tilde Q(\lambda') = \sum_{\substack{t, s \in T \\ t \ne s}} \lambda_s \lambda'_t f(v_s, v_t).
    \]
    We may think of $\lambda$ as an element of $K^\tv$, and $\tilde Q$ defines an orthogonal form on $K^\tv$ with corresponding symplectic form $\tilde f$.

    We aim to reduce the support of $\lambda$ as much as possible.
    Suppose $m\ge 10$, i.e., the support of $\lambda$ contains at least $10$ transvections $t_1, \dots, t_{10}$.
    Applying \Cref{prop:connect-up,prop:winkle}, $T_0 \cup \{t_1, \dots, t_{10}\}$ is contained in a bounded-cardinality nondegenerate set $T_1 \subset T$ such that $G_1 = \gen{T_1}$ acts irreducibly on $V_1 = V(T_1)$. Since $T_0 \subset T_1$, $G_1$ must act classically, and therefore $G_1$ must be an orthogonal group by hypothesis.
    Since $G_0 \le G_1$, by \Cref{main-thm-bounded-rank} we have
    \[
        \ell_T(G_1) \le \ell_{T_1 \cup G_0}(G_1) \ell_T(G_0) \le O(1).
    \]
    By \Cref{lem:Q-surjective-on-affine-subspace}, we can find a transvection $t \in G_1 \subset T^L$ (for some constant $L$) such that $t = 1 + v \otimes v^*$ for
    \[
        v \in \lambda_1 v_1 + \lambda_2 v_2 + K v_3 + \cdots + K v_{10}.
    \]
    Hence we split the given linear relation into two shorter linear relations
    \begin{align*}
        &v + \lambda_1 v_1 + \lambda_2 v_2 + \mu_3 v_3 + \cdots + \mu_{10} v_{10} &= 0, \\
        &v + (\lambda_3 + \mu_3) v_3 + \cdots + (\lambda_{10} + \mu_{10}) v_{10} + \sum_{i>10} \lambda_i v_i &= 0.
    \end{align*}

    Repeating this argument for $\floor{m/10}$ disjoint subsets of $\{t_1,\dots,t_m\}$ of size $10$, we find that we can add $\le m/10$ transvections $t_i = 1 + v_i \otimes v_i^* \in T^{O(1)}$ such that our original linear relation can be written as a sum of linear relations of support $\le 11$ and one of support at most $\ceil{9m/10}$, i.e., $\lambda = \lambda^{(1)} + \cdots + \lambda^{(s)}$ where each $\lambda^{(j)}$ has support size at most $\ceil {9m/10}$ and
    \[
        \sum_{t \in T^L} \lambda_t^{(j)} v_t = 0 \qquad (j=1,\dots,s).
    \]
    Now we repeat this argument for $T^L$ and each of the $\lambda^{(j)}$, etc, $\log n$ times.
    The end result is that $\lambda$ gets written as a sum of some $\mu^{(1)}, \dots, \mu^{(N)}$ where each vector $\mu$ has support size $O(1)$, and
    \[
        \sum_{t \in T^{L^{\log n}}} \mu^{(j)}_t v_t = 0 \qquad (j=1,\dots,N).
    \]
    Note that $L^{\log n} = n^{O(1)}$.
    Finally we use
    \begin{equation}
        \label{eq:tilde-Q}
        \tilde Q(\lambda) = \tilde Q \br{\sum_{j=1}^N \mu^{(j)}} = \sum_{j=1}^N \tilde Q(\mu^{(j)}) + \sum_{1 \le j < j' \le N} \tilde f(\mu^{(j)}, \mu^{(j')}).
    \end{equation}

    We claim that each term on the right-hand side of \eqref{eq:tilde-Q} is zero.
    Indeed, first consider some $\tilde Q(\mu)$, $\mu = \mu^{(j)}$.
    We know that $\mu$ is supported on a set of $O(1)$ transvections $t \in T^{L^{\log n}}$.
    By \Cref{prop:connect-up,prop:winkle}, the union of $T_0$ with the support of $\mu$ is contained in a bounded-cardinality nondegenerate set $T_1 \subset T^{L^{\log n}}$ such that $G_1 = \gen{T_1}$ acts irreducibly on $V_1 = [V, T_1]$.
    Since $T_0 \subset T_1$, $G_1$ acts classically and therefore as an orthogonal group on $V_1$.
    Applying the forward direction of \Cref{prop:orthogonal}, we find that $\tilde Q(\mu) = 0$.
    By the same argument applied to the union of the supports of $\mu$ and $\mu'$, $\mu' = \mu^{(j')}$, we also have $\tilde Q(\mu + \mu') = \tilde Q(\mu) + \tilde Q(\mu') + \tilde f(\mu, \mu') = 0$, and since $\tilde Q(\mu) = \tilde Q(\mu') = 0$ we have $\tilde f(\mu, \mu') = 0$.
    This proves the claim.
    Hence $\tilde Q(\lambda) = 0$ and the proof is complete.
\end{proof}

\subsection{Conclusion}

\begin{theorem}[Certification]
    \label{thm:certification}
    Let $n$ be sufficient large.
    Let $T \subset \tv \subset \SL_n(q)$ be a set of transvections such that $G = \gen T$ is an irreducible classical group. Then there is a subset $T_0 \subset T^{(n \log q)^{O(1)}}$ of cardinality $O(1)$ such that, for any subset $T_1 \subset G \cap \tv$ containing $T_0$ such that $\Gamma(T_1)$ is strongly connected,
    the group $G_1 = \gen{T_1}$ acts on the section $W = [V,G_1] / [V, G_1]^{G_1}$ as an irreducible classical group with the same type and defining field as $G$ itself, and moreover $\dim W > 10$.
\end{theorem}
\begin{proof}
    In brief, combine \Cref{prop:unitary-symplectic-certificate,prop:sym-type-certificate,prop:monomial-type-certificate,prop:subfield,prop:orthogonal-certificate}.

    We may assume $G$ has defining field $K = \F_q$.
    By \Cref{prop:density} and replacing $T$ with $T^{2n-1} \cap \tv$, we may assume $T$ is dense.
    By \Cref{prop:subfield}, there is a subset $S_0 \subset T^{(\log q)^{O(1)}}$ of cardinality $|S_0| \le 3$ such that $L(S_0) = K$.
    By \Cref{prop:unitary-symplectic-certificate} there is a set $S_1 \subset T$ of cardinality at most $10$ such that $S_1$ contains a non-symplectic cycle if $T$ does and $S_1$ contains a non-unitary cycle if $T$ does.
    By \Cref{prop:sym-type-certificate}, if $q=2$ and $G$ is contained in a symplectic group, there is a subset $S_2 \subset T$ of cardinality $|S_2| \le 2000$ that is not contained in a symmetric-type subgroup.
    By \Cref{prop:monomial-type-certificate}, if $G$ is not contained in a symplectic group, there is a subset $S_3 \subset T$ of cardinality $|S_3|\le 1000$ that is not contained in a monomial subgroup.
    By \Cref{prop:orthogonal-certificate}, if $G = \Sp_n(q)$ there is a subset $S_4 \subset T^{(n \log q)^{O(1)}}$ of cardinality $O(1)$ that is not contained in an orthogonal subgroup.
    Finally, by density of $T$ there is a bounded-cardinality set $T_0 \subset T$ such that for each $i$ (such that $S_i$ is defined) we have $S_i \subset T_0$, moreover the natural maps
    \begin{align*}
        V(S_i) \to V(T_0) / V(T_0) \cap V^*(T_0)^\perp, \\
        V^*(S_i) \to V^*(T_0) / V^*(T_0) \cap V(T_0)^\perp
    \end{align*}
    are injective, and moreover both spaces on the right have dimension at least $2000$.

    Let $T_1 \subset \tv$ be a set containing $T_0$ such that $\Gamma(T_1)$ is strongly connected. By \Cref{prop:irreducible-section}, $G_1 = \gen{T_1}$ acts on $W = [V,G_1] / [V,G_1]^{G_1}$ irreducibly. Let $\bar G_1 = G_1 | W$ and let $\bar T_1$ be the image of $T_1$ in $\bar G_1$.
    By \Cref{prop:irreducible-section}, $\Gamma(T_1)$ and $\Gamma(\bar T_1)$ are naturally identified and the weights of corresponding cycles are equal.
    In particular, $L(\bar T_1) = L(T_1) = K$, so $\bar G_1$ has defining field $K$.
    Also $\Gamma(\bar T_1)$ contains symplectic cycles if and only if $\Gamma(T_1)$ does, and ditto for unitary cycles.
    We also have $\dim W \ge 2000$.
    The symmetric-type and monomial-type certificates ensure that $\bar G_1$ has neither symmetric type nor monomial type.
    Thus $\bar G_1$ must have the same type as $G$.
\end{proof}

\section{Bounding the diameter}

Finally we turn to the proof of our main theorem, \Cref{thm:main}. Let $G$ be one of the groups in scope and let $X$ be a generating set for $G$ containing a transvection $t$.
Let
\[
    T_m = \{t^{x_1 \cdots x_m} : x_1, \dots, x_m \in X\}, \qquad G_m = \gen{T_m}.
\]
Then $G_0 \le G_1 \le \cdots $ is an increasing sequence of subgroups of $G$ such that $G_i = G_{i+1}$ implies $G_i = \gen{t^G} = G$.
Let $m$ be the first index such that $G_m = G$.
Then $2^m \le |G_m| = |G|$, so $m \le \log_2 |G|$, and $\gen {T_m} = G$.
This demonstrates that for the purpose of proving \Cref{thm:main} it is sufficient to assume $X \subset \tv$ (cf.~\cite{GHS}*{Lemma~2.1}).

Thus assume $T \subset \tv$ is a generating set for $G$.
In particular, $\gen T$ is irreducible.
Now we apply \Cref{thm:certification} to obtain a certificate $T_0 \subset T^{(n \log q)^{O(1)}}$ of cardinality $O(1)$.
We may assume $T_0$ is weakly nondegenerate and $\Gamma(T_0)$ is strongly connected.
Let $G_0 = \gen{T_0}$.
Unless $G = \SL(V)$, $T_0$ is in fact nondegenerate and $G_0$ is a classical group with the same type as $G$, defining field $K$, and domain $V_0 = [V,G_0]$ of dimension $\dim V_0 \le |T_0| = O(1)$.
By \Cref{main-thm-bounded-rank}, $\ell_{T_0}(G_0) \le O(\log q)^{O(1)}$.
By replacing $T$ with $T^{(n \log q)^{O(1)}} \cap \tv$ we may thus assume $T_0 \subset G_0 \cap \tv \subset T$.

If $G = \SL(V)$, $T_0$ may only be weakly nondegenerate, but $G_0 | W_0 = \SL(W_0)$ where $W_0 = V_0 / V_0^{G_0}$, $V_0 = [V,G_0]$, and $\dim W_0 > 10$.
By \Cref{main-thm-bounded-rank,prop:tv-in-SL-sections} we have $\ell_{T_0}(G_0) \le O(\log q)^{O(1)}$.
Thus again by replacing $T$ with $T^{(n\log q)^{O(1)}} \cap \tv$ we may assume $T_0 \subset G_0 \cap \tv \subset T$.

Let $V = \F_q^n$ be the natural module for $G$.
Then $V$ is a linear, unitary, symplectic, or characteristic-2 orthogonal space.
Let us call a vector $v \in V$ \emph{transvective} if
there is a transvection $t \in G \cap \tv$ such that $[V, t] = \gen v$.
If $V$ is linear or symplectic, all vectors are transvective;
if $V$ is unitary, a vector is transvective if and only if it is singular;
if $V$ is orthogonal, a vector is transvective if and only if it is nonsingular.

\begin{lemma}
    \label{lem:transvective-1}
    Let $v = \sum_{i=1}^k v_i$ for transvective vectors $v_1, \dots, v_k$.
    Then $v - \lambda v_i - \mu v_j$ is transvective for some $i, j \in \{1,\dots,k\}$ and $\lambda, \mu \in K$.
\end{lemma}
\begin{proof}
    Clearly we may assume $V$ is unitary or orthogonal, as otherwise all vectors are transvective.

    \emph{Unitary case.}
    (Cf.~\cite{GHS}*{Lemma~4.19}, which assumes $q$ is odd).
    We will take $\mu = 0$.
    If $v$ is singular then we can take $\lambda = 0$.
    If $v$ is nonsingular then
    \[
        f(v, v) = \sum_{i=1}^k f(v_i, v) \ne 0,
    \]
    so $f(v_i, v) \ne 0$ for some $i$ and
    \[
        f(v - \lambda v_i, v - \lambda v_i) = f(v, v) - \tr(\lambda f(v_i, v)).
    \]
    Since the trace map $\tr : \F_q \to \F_{\sqrt q}$ is surjective and $f(v,v) \in \F_{\sqrt q}$, we can choose $\lambda \in K$ so that $f(v,v) = \tr(\lambda f(v_i, v))$, and it follows that $v - \lambda v_i$ is singular, as required.

    \emph{Orthogonal case.}
    Observe that
    \[
        Q(v - \lambda v_i) = Q(v) - \lambda f(v, v_i) + \lambda^2 Q(v_i).
    \]
    Since $Q(v_i) \ne 0$, there are at most two $\lambda \in K$ for which this is zero, so we are done unless $q=2$, $Q(v) = 0$, and $f(v, v_i) = 1$ for all $i$.
    In this case there must be some $i, j$ such that $f(v_i, v_j) \ne 0$, and
    \[
        Q(v-v_i-v_j) = Q(v_i) + Q(v_j) + f(v,v_i) + f(v,v_j) + f(v_i,v_j) = 1,
    \]
    so we are done.
\end{proof}

\begin{lemma}
    \label{lem:transvective-2}
    Let $b_1, \dots, b_n$ be a basis for $V$ with $b_1, \dots, b_n$ transvective.
    For $v = \sum_{i=1}^n x_i b_i \in V$, let $s(v)$ be the number of nonzero coefficients $x_i$.
    Then any transvective vector $v$ with $s(v) = k$ is the sum of $4$ transvective vectors $v_1, v_2, v_3, v_4$ with $s(v_i) \le k/2 + 2$ for each $i$.
\end{lemma}
\begin{proof}
    Clearly we may write $v = u_1 + u_2$ for some (not necessarily transvective) $u_1, u_2 \in V$ such that $s(u_1) \le k/2+1$ and $s(u_2) \le k/2$.
    By \Cref{lem:transvective-1}, there are indices $i,j$ and $\lambda, \mu \in K$ such that $v_1 = u_1 - \lambda b_i - \mu b_j$ is transvective and $s(v_1) \le s(u_1) \le k/2+1$.
    Let $u_2' = u_1 + \lambda b_i + \mu b_j$.
    Applying \Cref{lem:transvective-1} again, there are indices $i', j'$ and $\lambda', \mu' \in K$ such that $v_2 = u_2' - \lambda' b_{i'} - \mu' b_{j'}$ is transvective and $s(v_2) \le s(u_2') \le k/2+2$.
    Let $v_3 = \lambda' b_{i'}$ and $v_4 = \mu' b_{j'}$.
    Then $v = v_1 + v_2 + v_3 + v_4$, by construction $v_1, v_2, v_3, v_4$ are transvective, and $s(v_1) \le k/2+1$, $s(v_2) \le k/2+2$, and $s(v_3), s(v_4) \le 1$.
\end{proof}

\begin{proposition}
    $T^{n^{O(1)}} \supset G \cap \tv$.
\end{proposition}
\begin{proof}
    \emph{Case $G \ne \SL_n(q)$}.
    Assume $V$ is a symplectic, unitary, or orthogonal space.
    Let $f$ be the invariant form and for each $v \in V$ let $v^*$ be the dual element associated to $v$ by $f$.
    Recall that a vector $v \in V$ is transvective if and only if there is a transvection $t \in G \cap \tv$ such that $t = 1 + \lambda v \otimes v^*$ for some $\lambda \in K$.
    By \Cref{prop:irreducible}, there is a (transvective) basis $b_1, \dots, b_n$ for $V$ such that $1 + \mu_i b_i \otimes b_i^* \in T$ for some $\mu_i \in K$ for each $i = 1, \dots, n$.
    Let $s : V \to \{0,\dots,n\}$ be the support function defined in \Cref{lem:transvective-2}
    and for each $k \ge 1$ let $T^{(k)}$ be the set of transvections $t = 1 + \lambda v \otimes v^* \in G \cap \tv$ such that $s(v) \le k$.
    It suffices to prove that $T^{(5)} \subset T^{O(1)}$ and $T^{(2k)} \subset (T^{(k+2)})^{O(1)}$ for each $k \ge 1$, for then it follows that
    \[
        G \cap \tv = T_n \subset T^{O(1)^{\ceil {\log_2 n}}} = T^{n^{O(1)}}.
    \]

    Let $t \in T^{(5)}$, say $t = 1 + \lambda  v \otimes v^*$ and $v = \sum_{i=1}^k x_i b_i$, $k \le 5$.
    By \Cref{prop:connect-up,prop:winkle}, $T_0 \cup \{1 + \mu_i b_i \otimes b_i^* : 1 \le i \le k\}$ is contained in a nondegenerate set $T_1 \subset T$ of cardinality $O(1)$ such that $\Gamma(T_1)$ is strongly connected.
    Then $G_1 = \gen{T_1}$ is a classical group with the same type as $G$ and defining field $K$ and domain $V_1 = [V,G_1]$.
    Since $b_1, \dots, b_k \in V_1$ it follows that $v \in V_1$ and $t \in G_1$.
    By \Cref{main-thm-bounded-rank} and the fact that $T$ contains $G_0 \cap \tv$ and $\log |G_0 \cap \tv| \ge \log q$, we have $G_1 \subset T^{O(1)}$.
    Thus $T^{(5)} \subset T^{O(1)}$.

    Similarly, let $k \ge 1$ and let $t \in T^{(2k)}$, say $t = 1 + \lambda v \otimes v^*$ where $s(v) \le 2k$.
    By \Cref{lem:transvective-2}, $v = v_1 + v_2 + v_3 + v_4$ where $s(v_i) \le k + 2$ for each $i$.
    Let $t_i \in T^{(k+2)}$ be a transvection of the form $1 + \lambda_i v_i \otimes v_i^*$ for $i = 1,2,3,4$.
    By \Cref{prop:connect-up,prop:winkle}, $T_0 \cup \{1 + \lambda_i v_i \otimes v_i^* : i = 1,2,3,4\}$ is contained in a nondegenerate set $T_1 \subset T^{(k+2)}$ of cardinality $O(1)$ such that $\Gamma(T_1)$ is strongly connected.
    Then $G_1 = \gen{T_1}$ is a classical group with the same type as $G$ and defining field $K$ and domain $V_1 = [V,G_1]$.
    Since $v_1,v_2,v_3,v_4 \in V_1$, it follows exactly as above using \Cref{main-thm-bounded-rank} that $t \in G_1 \subset (T^{(k+2)})^{O(1)}$.
    Thus $T^{(2k)} \subset (T^{(k+2)})^{O(1)}$, as required.

    \emph{Case $G = \SL_n(q)$.} Now all vectors are transvective.
    Again by \Cref{prop:irreducible} there are elements $1 + b_i \otimes f_i \in T$ ($1 \le i \le n$) such that $b_1, \dots, b_n$ is a basis for $V$, as well as elements $1 + c_i \otimes g_i \in T$ ($1 \le i \le n$) such that $g_1, \dots, g_n$ is a basis for $V^*(T)$.
    Define $s(v)$ exactly as before with respect to the basis $b_1, \dots, b_n$ and $s(\phi)$ analogously for $\phi \in V^*$ with respect to the basis $g_1, \dots, g_n$.
    For each $k \ge 1$ let $T^{(k)}$ be the set of all transvections $t = 1 + v \otimes \phi \in \tv$ such that $s(v) \le k$ and $s(\phi) \le k$.
    It suffices to prove that $T^{(1)} \subset T^{O(1)}$ and $T^{(2k)} \subset (T^{(k)})^{O(1)}$ for each $k \ge 1$.

    Let $t \in T^{(1)}$, so without loss of generality $t = 1 + \lambda b_i \otimes g_j$ for some indices $i,j$ and $\lambda \in K$.
    By \Cref{prop:connect-up,prop:winkle}, $T_0 \cup \{1 + b_i \otimes f_i, 1 + c_j \otimes g_j\}$ is contained in a bounded-cardinality weakly nondegenerate set $T_1 \subset T$ such that $\Gamma(T_1)$ is strongly connected.
    Applying \Cref{prop:tv-in-SL-sections} and \Cref{main-thm-bounded-rank}, $t \in G_1 = \gen{T_1}$ and $G_1 \subset T^{O(1)}$.
    Since $t$ was arbitrary this proves $T^{(1)} \subset T^{O(1)}$.

    Similarly, let $k \ge 1$ and let $t \in T^{(2k)}$, say $t = 1 + v \otimes \phi$ where $s(v) \le 2k$ and $s(\phi) \le 2k$.
    Write $v=v_1+v_2$ and $\phi = \phi_1+\phi_2$ where $s(v_1),s(v_2),s(\phi_1),s(\phi_2) \le k$.
    Let $t_i = 1 + v_i \otimes \phi_i$ for $i = 1,2$ and observe that $t_1, t_2 \in T^{(k)}$.
    By \Cref{prop:connect-up,prop:winkle}, $T_0 \cup \{t_1,t_2\}$ is contained in a weakly nondegenerate set $T_1 \subset T^{(k)}$ of cardinality $O(1)$ such that $\Gamma(T_1)$ is strongly connected.
    Observe that $v = v_1 + v_2 \in V(T_1)$ and $\phi = \phi_1 + \phi_2 \in V^*(T_1)$, so $t \in G_1 = \gen{T_1}$.
    Exactly as before we have $G_1 \subset (T^{(k)})^{O(1)}$, so we have proved that $T^{(2k)} \subset (T^{(k)})^{O(1)}$, as required.
\end{proof}

\begin{proposition}
    $\ell_T(G) \le n^{O(1)}$.
\end{proposition}
\begin{proof}
    By the previous proposition, $\ell_T(G \cap \tv) \le n^{O(1)}$.
    By the Liebeck--Shalev theorem~\cite{liebeck-shalev}*{Theorem~1.1} we have $\ell_{G \cap \tv}(G) \le O(n)$.
    Thus
    \[
        \ell_T(G) \le \ell_T(G \cap \tv) \ \ell_{G \cap \tv}(G) \le n^{O(1)},
    \]
    as claimed.
\end{proof}

This completes the proof of \Cref{thm:main}, since $\ell_X(T) \le (n \log q)^{O(1)}$.

\bibliography{refs}
\end{document}